\documentclass[12pt,twoside]{amsart}
\usepackage{amsthm}
\usepackage{mathrsfs}
\usepackage{latexsym}
\usepackage{amsmath}
\usepackage{amssymb}
\usepackage{color}
\usepackage{tikz}
\usetikzlibrary{arrows}
\usepackage{comment}
\usepackage{braket}

\usepackage{amsmath}
\usepackage{amssymb}
\usepackage{makecell}
\usepackage{multirow}
\usepackage{longtable}
\usepackage{scalerel}
\usepackage{stackengine}
\usepackage{hyperref}
\hypersetup{
    colorlinks = true,
    linkcolor = blue,
    urlcolor = blue,
    citecolor = blue,
    anchorcolor = black,
    pdftoolbar = true
    linkbordercolor = {white},
}

\usepackage{fancyhdr}

\stackMath
\def\hatgap{1pt}
\def\subdown{-2pt}
\newcommand\reallywidehat[2][]{%
	\renewcommand\stackalignment{l}%
	\stackon[\hatgap]{#2}{%
		\stretchto{%
			\scalerel*[\widthof{$#2$}]{\kern-.6pt\bigwedge\kern-.6pt}%
			{\rule[-\textheight/2]{1ex}{\textheight}}
		}{0.5ex}
		_{\smash{\belowbaseline[\subdown]{\scriptstyle#1}}}%
}}

\usetikzlibrary{positioning,calc}
\usetikzlibrary{calc}
\usetikzlibrary{shapes.geometric}
\usetikzlibrary{arrows}
\usetikzlibrary{decorations.pathmorphing}
\usetikzlibrary{decorations.text}
\usetikzlibrary{decorations.shapes}
\usetikzlibrary{decorations.fractals}
\usetikzlibrary{decorations.footprints}
\newcommand{\be}{\begin{equation}}
\newcommand{\ee}{\end{equation}}

\newtheorem{theorem}{Theorem}[section]

\newtheorem{lemma}[theorem]{Lemma}

\newtheorem{question}[theorem]{Question}
\newtheorem{conjecture}[theorem]{Conjecture}

\newcommand{\norm}[1]{\left\Vert #1\right \Vert}
\newcommand{\size}[1]{\fontsize{10pt}{\baselineskip}\selectfont{#1}}

\newcommand{\mn}{/-3}
\newcommand{\nn}{/3}
\newcommand{\fqudit}[5]{
\node at (#1-#3,#2+#4) {\size{$#5$}};
\draw (#1,#2) --(#1,#2+#4)  arc (0:-180:#3) -- (#1-#3-#3,#2);
}

\newcommand{\fdoublequdit}[6]
{
\fqudit{#1}{#2}{3*#3}{#4}{}
\fqudit{#1-2*#3}{#2}{#3}{#4}{#5}
\node at (#1-3*#3, #2+#4+2*#3) {\size{$#6$}};
}

\newcommand{\fmeasure}[5]{
\node at (#1-#3,#2-#4) {\size{$#5$}};
\draw (#1,#2) --(#1,#2-#4)  arc (0:180:#3) -- (#1-#3-#3,#2);
}

\newcommand{\FS}{\mathfrak{F}}

\newcommand{\bR}{\mathbb{R}}
\newcommand{\bN}{\mathbb{N}}

\newcommand{\fA}{\mathfrak{A}}

\newcommand{\cN}{\mathcal{N}}
\newcommand{\cM}{\mathcal{M}}

\newcommand{\cC}{\mathcal{C}}

\newcommand{\cA}{\mathcal{A}}
\newcommand{\cB}{\mathcal{B}}

\newcommand{\sP}{\mathscr{P}}
\newcommand{\fF}{\mathfrak{F}}

\renewcommand{\geq}{\geqslant}
\renewcommand{\leq}{\leqslant}

\newcommand{\tr}{tr} 

\makeatletter
\def\l@section{\@tocline{1}{0pt}{1pc}{}{}}
\def\l@subsection{\@tocline{2}{0pt}{1pc}{4.6em}{}}
\def\l@subsubsection{\@tocline{3}{0pt}{1pc}{7.6em}{}}
\renewcommand{\tocsection}[3]{%
  \indentlabel{\@ifnotempty{#2}{\makebox[2.3em][l]{%
    \ignorespaces#1 #2.\hfill}}}#3}
\renewcommand{\tocsubsection}[3]{%
  \indentlabel{\@ifnotempty{#2}{\hspace*{2.3em}\makebox[2.3em][l]{%
    \ignorespaces#1 #2.\hfill}}}#3}
\renewcommand{\tocsubsubsection}[3]{%
  \indentlabel{\@ifnotempty{#2}{\hspace*{4.6em}\makebox[3em][l]{%
    \ignorespaces#1 #2.\hfill}}}#3}
\makeatother

\title{Quantum Fourier Analysis}
\author{Arthur Jaffe}
\address{Harvard University, Cambridge, MA 02138}
\email{jaffe@g.harvard.edu}
\author {Chunlan Jiang}
\address {Hebei Normal University} 
\email{cljiang@hebtu.edu.cn}
\author {Zhengwei Liu}
\address {Tsinghua University}
\email {liuzhengwei@mail.tsinghua.edu.cn}
\author {Yunxiang Ren}
\address{Harvard University, Cambridge, MA 02138}
\email {yren@g.harvard.edu}
\author{Jinsong Wu}
\address {Institute for Advanced Study in Mathematics, Harbin Institute of $\phantom{xxTT}$ Technology}
\email {wjs@hit.edu.cn}

\begin{document}

\begin{abstract}
{\em Quantum Fourier analysis} is a new subject that combines an algebraic Fourier transform (pictorial in the case of subfactor theory) with analytic estimates.  This provides interesting tools to investigate phenomena such as quantum symmetry.  We establish bounds on the quantum Fourier transform $\FS$, as a map between suitably defined $L^{p}$ spaces, leading to a new uncertainty principle for relative entropy. We cite several applications of the quantum Fourier analysis in subfactor theory, in category theory, and in quantum information. We suggest a new topological inequality, and we outline several open problems.
\end{abstract}

\maketitle

\tableofcontents
In this paper we explore {\em quantum Fourier analysis} (QFA), a new subject revolving around the study of Fourier analysis of quantum symmetries. The discovery of such symmetries emerged in the 1970's, and has flourished ever since. It represents a major advance both in mathematics and in physics, as well as in the relation between these two subjects. Thus QFA adds a new dimension to the over 200-year old subject of classical Fourier analysis (CFA), analyzing the Fourier transform $F$.

CFA led to insights and solutions of problems in almost every field of mathematics, including partial differential equations, probability theory, number theory, representation theory, topology, geometry, etc.  It ultimately led to the categorization of Fourier duality~\cite{Tannaka,Krein}. The Hausdorff-Young inequality is a bound on the norm $M_{p}=\norm{F}_{L^{p}\to L^{q}}$, where $q=p/(p-1)$. 
Hirschman discovered that differentiating $M_{p}$ gives an uncertainty principle for the Shannon entropy, generalizing the well-known Heisenberg principle. He and Everett conjectured the optimal inequality~\cite{Hirschman,Everett}. Deep and beautiful proofs were found~\cite{Beckner, BiaMyc75, BrascampLieb}.  

Classical hypercontractivity states  $\norm{e^{-tH}}_{L^{2}\to L^{q}}\leqslant 1$, where  $H$ is a simple harmonic oscillator Hamiltonian with unit angular frequency and $e^{{2t}}\geqslant q-1\geqslant 1$~\cite{Nelson1,Glimm,Nelson2}.  The classical Hausdorff-Young inequality is a consequence of $F=e^{i\pi H/2}$~\cite{Beckner}. Further inequalities can be found in many papers as~\cite{Federbush,Gross,Carlen-Lieb,Stein,Janson,Xu,Junge,King}, suggesting in retrospect a  bridge from CFA to QFA. 

\section{Quantum Fourier Analysis}
A quantum Fourier transform $\mathfrak{F}$ defines Fourier duality between quantum symmetries, which could be analytic, algebraic, geometric, topological and categorical. The quantum symmetries could be finite or infinite, discrete or continuous, commutative or non-commutative. 
In certain contexts $\mathfrak{F}$ can be defined pictorially---as in the picture language program~\cite{JL18}.  QFA is the study of structures involving $\mathfrak{F}$.  

It is possible to estimate various norms $\norm{\mathfrak{F}}_{L^{p}\to L^{q}}$ as transformations between non-commutative $L^p$ spaces, and results in~\cite{Liu16, JLW16} represent early breakthroughs in the application and formulation of the new subject of QFA. As QFA is  more sophisticated than CFA, these subjects have differences as well as similarities; we explore them both.  

Let us consider an example of similarities and differences. In CFA, the extremizers of the Hausdorff-Young inequality (and many others) are Gaussians. In QFA on subfactors, the Hausdorff-Young also holds. The extremiziers are bishifts of biprojections. A biprojection is a projection whose quantum Fourier transform is a multiple of a projection \cite{Bisch}. So the behavior under Fourier transformation of the extremizers in QFA on subfactors are similar to those in CFA, while their algebraic properties differ.

In this paper we give a unified view of QFA. We establish a new ``relative'' inequality between pictures that yields an uncertainty principle for relative entropy. We propose a new universal quantum inequality, namely~\eqref{eq:top-braslieb}, that  unifies many other quantum inequalities. This is similar to the way the Brascamp-Lieb inequality unites Young's inequality,  H\"{o}lder inequality, and others, in CFA. Throughout the paper we cite  applications of QFA. Finally in \S\ref{Sect:Questions}, we state some general goals for the future and some open questions. 

QFA reveals insight and intrinsic structure, as well as relations between, fusion rings, fusion categories, and subfactors. We show how  the ``Schur product property'' provides a powerful obstruction to distinguish mathematical objects. QFA also provides a new approach to quantum entanglement, uncertainty relations, and other problems in quantum information. We are certain that QFA will lead to other advances in many different fields.


\section{QFA on Fusion Rings\label{subsec:fusion rings}}
Let us start with fusion rings, as introduced by Lusztig~\cite{Lus87}; this is an interesting quantum symmetry beyond groups. See \cite{EGNO15} for further results and references. 
A fusion ring $\fA$ is a ring which is free as a $\mathbb{Z}$-module, with a basis $\{x_1, x_2, \ldots, x_m\}, m\in\mathbb{N}$, with $x_{1}=1$, and such that

[$\mathcal{B}$1] $x_jx_k=\sum_{s=1}^m N_{j,k}^s x_s$, with $N_{j,k}^s\in \bN$, and 

[$\mathcal{B}$2] there exists an involution $*$ on $\{1, 2, \ldots, m\}$ such that $N_{j, k}^1=\delta_{j, k^*}$ inducing an anti-isomorphism of $\fA$, given by $x_k^*:=x_{k^*}$ and 
$x_k^* x_j^*=(x_jx_k)^*$.

From a fusion ring, one can construct two $C^*$-algebras with faithful tracial states, and a unitary Fourier transform between them. 
We recall this construction and the quantum Fourier analysis on fusion rings as studied in \cite{LPW19}. Define a unital $^*$-algebra over $\mathbb{C}$ with basis $\{x_j\}$, with  multiplication given by property [$\cB$1], and with adjoint given by $^*$ in [$\cB$2]. Define a linear functional $\tau$ given by the Dirac measure on the basis, $\tau(x_j)=\delta_{j,1}$. From $[\cB2]$ we infer that $\tau$ is a strictly-positive trace and defines an inner product. This gives the $C^*$-algebra $\cB$ by the  Gelfand-Naimark-Segal construction.


One can define a second (in this case abelian) $C^*$-algebra $\cA$ from $\fA$ with multiplication $\diamond$, adjoint $\#$, given by another strictly positive linear functional $d$ on $\cA$.  Let  

[$\cA$1] $x_j \diamond x_k =\delta_{j,k} \,d(x_j)^{-1} x_j$,

[$\cA$2] $x_j^{\#}=x_j$.  

\noindent  
Here $d(\cdot)$ is defined linear, by sending a basis element $x_j$ to $d(x_j)$, the operator norm of the fusion matrix $M_j$, with entries $(M_j)_{s,t}=N_{j,s}^{t}$. This is called the Perron-Frobenius dimension of $x_j$. 
The trace $d$ (resp. $\tau$) on the finite-dimensional $C^*$-algebra $\cA$ (resp. $\cB$)  defines   $L^p$ norms on $\cA$ (resp. $L^q$ norms on $\cB$) by, 
\begin{equation*}
    \norm{a}_{\cA_p}=d(|a|^p)^{1/p}\;, \qquad\norm{b}_{\cB_q}=\tau(|b|^q)^{1/q}\;.
\end{equation*}

Then $\cA$ and $\cB$ are two $C^*$-algebras with the same basis. 
We use the classical notation $\fF:x\mapsto \widehat{x}$ for the \textit{Fourier transform} as the linear map from $\cA$ to $\cB$ defined by 
$\fF(x_j)=x_j$.  
The Fourier transform can be extended to a map from $L^p(\cA,d)$ to $L^q(\cB,\tau)$ for $1/p+1/q=1$.  
Plancherel's formula follows as: $\|\fF(x)\|_{\cB_{2}}=\|x\|_{\cA_{2}}$\;.

We now summarize several theorems proved in \cite{LPW19} about QFA  on fusion rings, including the quantum Schur product theorem (QSP), the quantum Hausdorff-Young inequality (QHY) with $1/p+1/q=1$, the quantum Young inequality (QY) with $1/p+1/q=1+1/r$, and the basic quantum uncertainty principles (QUP)---defined in terms of the von Neumann entropy,
\begin{align*}
 H_{\cA}(|x|^2)&=-d((x^{\#}\diamond x)\diamond\log(x^{\#}\diamond x));\\  H_{\cB}(|{\widehat x}|^2)&=-\tau((\widehat{x}^* \widehat{x})\log(\widehat{x}^*\widehat{x})), 
\end{align*}
and in terms of the support $\mathcal{S}_{\cA}(x)=d(R(x))$, $\mathcal{S}_{\cB}(\widehat{x})=\tau(R(\widehat{x}))$, where $R(\cdot)$ is the range projection.
\begin{theorem} \label{Thm:Fusion}
For non-zero $x,y\in\cA$, 

\noindent
$[\text{QSP}]:$ $\fF^{-1}({\widehat x}{\widehat y}) \geq_{\cA} 0$, whenever both $ x,y\geq_{\cA} 0$.

\noindent
$[\text{QHY}]:$ $\| \widehat x\|_{\cB_q}\leq \|x\|_{\cA_p}$, for $1\leq p\leq 2$.

\noindent
$[\text{QY}]:$ $\| \fF^{-1}({\widehat x}{\widehat y})\|_{\cA_r}\leq \|x\|_{\cA_p}\|y\|_{\cA_q}$, for $1\leq p,q,r\leq \infty$.

\noindent
$[\text{QUP-}1]:$ 
$
H_{\cA}(|x|^2)
+ H_{\cB}(|\widehat x|^2)
+ 2\|x\|_{\cA_2}^2\log \|x\|_{\cA_2}^2
\geq 0\;.
$

\noindent
$[\text{QUP-}2]:$ $\mathcal{S}_{\cA}(x)\mathcal{S}_{\cB}(\widehat {x})\geq 1$.
\end{theorem}



\section{QFA on Subfactors} Modern subfactor theory was initiated by V. Jones in 1983 by his remarkable index theorem \cite{Jon83}. A subfactor of type $\mathrm{II}_1$ is an inclusion of $\mathrm{II}_1$ factors $\cN\subset \cM$ and its index describes the relative size of these two factors. Jones' index theorem asserts that the index $\delta^{2}$ of a subfactor belongs to the set $\{4\cos^2(\frac{\pi}{n}):n=3,4,\cdots\}\cup[4,\infty]$ and every possible value can be realized as the index of a subfactor. Subfactor theory turns out to be a natural framework to study quantum symmetry in statistical physics and quantum field theory, see e.g. \cite{EvaKaw98}.
Subfactor planar algebras~\cite{Jon99} provide a pictorial tool to study  subfactor theory. A planar algebra $\sP_\bullet=\{\sP_{n,\pm}:n\geq 0\}$ is a family of finite-dimensional vector spaces with an action of the operad of \textit{planar tangles}, similar to \textit{topological quantum field theory}~\cite{Atiyah88}. One represents an element in $\sP_{n,\pm}$ (called an $n$-box) by a labelled rectangle with $2n$ strings attached to its boundary. Each vector space $\sP_{n,\pm}$ is equipped with an involution $^*$,
which is compatible with the invloution of planar tangles. This involution $^*$ is called the adjoint, and given pictorially by reflection. 

Any $w\in\sP_{1,+}$ satisfies the spherical condition, and any $x\in\sP_{n,+}$ satisfies the ``reflection-positivity''  condition, 
\begin{equation}\label{equ: relative positivity}
\raisebox{-.25cm}{
	\begin{tikzpicture}
	\draw [fill=lightgray] (.25,.5) arc[radius=.1, start angle=180, end angle=90]--(.5,.6) arc[radius=.1, start angle=90, end angle=0]--(.6,0) arc[radius=.1, start angle=360, end angle=270]--(.35,-.1) arc[radius=.1, start angle=270, end angle=180];
	\draw [fill=white] (.05,0) rectangle (.45,.5);
	\node at (.25,.25) {$w$};
	\node at (0,.25) {\tiny $\$$};
	\end{tikzpicture}}=
\raisebox{-.3cm}{
	\begin{tikzpicture}
	\path [fill=lightgray] (-.15,-.15) rectangle (.55,.65);
	\begin{scope}[xscale=-1,shift={(-.5,0)}]
	\draw [fill=white] (.25,.5) arc[radius=.1, start angle=180, end angle=90]--(.5,.6) arc[radius=.1, start angle=90, end angle=0]--(.6,0) arc[radius=.1, start angle=360, end angle=270]--(.35,-.1) arc[radius=.1, start angle=270, end angle=180];
	\end{scope}
	\draw [fill=white] (.05,0) rectangle (.45,.5);
	\node at (.25,.25) {$w$};
	\node at (0,.25) {\tiny $\$$};
	\end{tikzpicture}}\;,\qquad
\raisebox{-.4cm}{
\begin{tikzpicture}
\draw (0,0) rectangle (.9,.35);
\draw [fill=yellow] (0,-.25) rectangle (.9,-.6);
\draw [fill=lightgray] (.05,0) rectangle (.15,-.25);
\draw [fill=lightgray] (.25,0) rectangle (.35,-.25);
\draw [fill=lightgray] (.75,0) rectangle (.85,-.25);
\node at (.28,-.125)[right] {\tiny $\cdots$};
\node at (.5,.25) {$x^{\phantom{*}}$};
\node at (.5,-.4) {$x^*$};
\node at (-.08,.125) {\tiny $\$$};
\node at (-.08,-.375) {\tiny $\$$};
\end{tikzpicture}}\geq 0\;.
\end{equation}
The action of planar tangles turns $\sP_{n,\pm}$ into $C^*$-algebras, and the trace gives a Hilbert-space representation by the GNS construction. We set $\cA=\sP_{2,+}$ and $\cB=\sP_{2,-}$. For elements $\cA$, one has pictorial representations for $x$, multiplication $xy$, the Fourier transform $\fF(x)$, the trace $\tr(x)$ as follows,
\vskip -.3cm
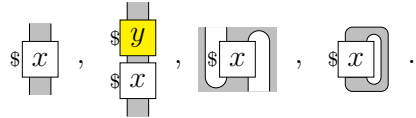
\begin{equation*}\label{Pictures-PA}
\scalebox{.95}{$
\raisebox{-.45cm}{
\begin{tikzpicture}
	\path [fill=lightgray] (.1,.75) rectangle (.4,-.25);
\draw (.1,.75)--(.1,-.25);
\draw (.4,.75)--(.4,-.25);
\draw [fill=white] (0,0) rectangle (.5,.5);
\node at (.25,.25) {$x$};
\node at (-.1,.25) {\tiny $\$$};
\end{tikzpicture}
}\;,
\raisebox{-.75cm}{
\begin{tikzpicture}
\path [fill=lightgray] (.1,-.55) rectangle (.4,1.05);
\draw (.1,-.55)--(.1,1.05);
\draw (.4,-.55)--(.4,1.05);
\draw [fill=white] (0,.2) rectangle (.5,-.3);
\draw [fill=yellow] (0,.3) rectangle (.5,.8);
\node at (.25,-.05) {$x$};
\node at (.25,.55) {$y$};
\node at (-.1+.03,-.05) {\tiny $\$$};
\node at (-.1+.03,.55) {\tiny $\$$};
\end{tikzpicture}
}\;,
\raisebox{-.4cm}{
\begin{tikzpicture}
\path[fill=lightgray] (-.3,.7) rectangle (.8,-.2);
\draw [fill=white] (.1,.7)--(.1,0) arc[radius=.15, start angle=360, end angle=180]--(-.2,.7);
\draw [fill=white] (.4,-.2)--(.4,.5) arc[radius=.15, start angle=180, end angle=0]--(.7,-.2);
\draw [fill=white] (0,0) rectangle (.5,.5);
\node at (.25,.25) {$x$};
\node at (-.1+.003,.25) {\tiny $\$$};
\end{tikzpicture}
}\;,
\raisebox{-.4cm}{
\begin{tikzpicture}
\draw [fill=lightgray] (.1,.5)--(.1,.6) arc[radius=.1, start angle=180, end angle=90]--(.6,.7) arc[radius=.1, start angle=90, end angle=0]--(.7,-.1) arc[radius=.1, start angle=0, end angle=-90]--(.2,-.2) arc[radius=.1, start angle=270, end angle=180]--(.1,0);
\draw [fill=white] (.4,.5) arc[radius=.1, start angle=180, end angle=0]--(.6,0) arc[radius=.1, start angle=360, end angle=180]--(.1,0);
\draw [fill=white] (0,0) rectangle (.5,.5);
\node at (.25,.25) {$x$};
\node at (-.1+.03, .25) {\tiny $\$$};
\end{tikzpicture}
}
$}\;.
\end{equation*}
\vskip -.2cm
Define the convolution product $*$ on $\cA$ by 
\vskip -.3cm
\begin{equation}
   \hskip -.2cm\raisebox{-.35cm}{\scalebox{.75}{
\begin{tikzpicture}
	\path [fill=lightgray](.4,.5) arc [radius=.2, start angle=180, end angle=0]--(.8,0) arc[radius=.2, start angle=360, end angle=180]--(.4,.25)--(.1,.25)--(.1,-.3)--(1.1,-.3)--(1.1,.8)--(.1,.8)--(.1,.5);
\draw (.1,.8)--(.1,-.3);
\draw (1.1,.8)--(1.1,-.3);
\draw (.4,.5) arc[radius=.2, start angle=180, end angle=0];
\draw (.4,0) arc[radius=.2, start angle=180, end angle=360];
\draw [fill=white] (0,0) rectangle (.5,.5);
\draw [fill=yellow] (.7,0) rectangle (1.2,.5);
\node at (-.1+.03,.25) {\tiny $\$$};
\node at (.63,.25) {\tiny $\$$};
\node at (.25,.25) {$x$};
\node at (.95,.25) {$y$};
\end{tikzpicture}}
}=x*y=\fF\left(\fF^{-1}(x)\fF^{-1}(y)\right)\;.
\end{equation}
The $C^*$-algebra $\cB$ has a similar pictorial representation. These pictures not only make quantum Fourier analysis transparent; they also provide a precise framework for proofs.         
\begin{theorem}[Schur product theorem, Theorem 4.1 of~\cite{Liu16}]\label{thm: Schur Product}
For any $0\leq x,y\in\cA$ (or $0\leq x,y\in\cB$), one has $0\leq x*y$. 
\end{theorem}

\begin{proof}
Since $\tr$ is faithful, it suffices to show that $\tr\left((x*y)z\right)\geq 0$ for any $0\leq z\in \cA$. Since $\cA$ is a $C^*$-algebra, there exist elements $x^{\frac{1}{2}},y^{\frac{1}{2}},z^{\frac{1}{2}}\in\cA$, such that
\begin{align}
\tr((x*y)z)=&
\raisebox{-.55cm}{
\begin{tikzpicture}
\draw [fill=lightgray] (.1,.4)--(.45,.8)--(.45,1.2) arc[radius=.1, start angle=180, end angle=90]--(1.3,1.3) arc[radius=.1, start angle=90, end angle=0]--(1.4,-.3) arc[radius=.1, start angle=360, end angle=270]--(.55,-.4) arc[radius=.1, start angle=270, end angle=180]--(.45,-.2)--(.1,.1);
\draw [fill=white] (1.1,.4)--(.75,.8)--(.75,1.1) arc[radius=.1, start angle=180, end angle=90]--(1.2,1.2) arc[radius=.1, start angle=90, end angle=0]--(1.3,-.2) arc[radius=.1, start angle=360, end angle=270]--(.85,-.3) arc[radius=.1, start angle=270, end angle=180]--(1.1,.1);
\draw [fill=white] (.4,.4) arc[radius=.1, start angle=180, end angle=90]--(.7,.5) arc[radius=.1, start angle=90, end angle=0]--(.8,.1) arc[radius=.1, start angle=360, end angle=270]--(.5,0) arc[radius=.1, start angle=270, end angle=180];
\draw [fill=white] (0,.1) rectangle (.5,.4);
\node at (.25,.25) {$x$};
\node at (-.04,.25) {\tiny $\$$};
\draw [fill=yellow] (.7,.1) rectangle (1.2,.4);
\node at (.66,.25) {\tiny $\$$};
\node at (.95,.25) {$y$};
\draw [fill=cyan] (.35,.8) rectangle (.85,1.1);
\node at (.35-.04,.95) {\tiny $\$$};
\node at (.6,.95) {$z$};
\end{tikzpicture}}=
\raisebox{-.25cm}{
\begin{tikzpicture}
\draw [fill=lightgray] (.1,.4)--(.1,.5) arc[radius=.1, start angle=180, end angle=90]--(1.7,.6) arc[radius=.1, start angle=90, end angle=0]--(1.8,0) arc[radius=.1, start angle=360, end angle=270]--(.2,-.1) arc[radius=.1, start angle=270, end angle=180]--(.1,.1);
\draw [fill=white] (.4,.4) arc[radius=.1, start angle=180, end angle=90]--(.7,.5) arc[radius=.1, start angle=90, end angle=0]--(.8,.1) arc[radius=.1, start angle=360, end angle=270]--(.5,0) arc[radius=.1, start angle=270, end angle=180];
\draw [fill=white] (1,.4) arc[radius=.1, start angle=180, end angle=90]--(1.4,.5) arc[radius=.1, start angle=90, end angle=0]--(1.5,.1) arc[radius=.1, start angle=360, end angle=270]--(1.1,0) arc[radius=.1, start angle=270, end angle=180];
\draw [fill=white] (0,.1) rectangle (.5,.4);
\node at (.25,.25) {$x$};
\node at (-.04,.25) {\tiny $\$$};
\draw [fill=yellow] (.7,.1) rectangle (1.2,.4);
\node at (.66,.25) {\tiny $\$$};
\node at (.95,.25) {$y$};
\draw [fill=cyan] (1.4,.1) rectangle (1.9,.4);
\node at (1.94,.25) {\tiny $\$$};
\node at (1.65,.25) {$z$};
\end{tikzpicture}}\\
=&
\raisebox{-.5cm}{
\begin{tikzpicture}
\begin{scope}
\draw [fill=lightgray] (.1,.4)--(.1,.5) arc[radius=.1, start angle=180, end angle=90]--(1.7,.6) arc[radius=.1, start angle=90, end angle=0]--(1.8,.4);;
\draw [fill=white] (1,.4) arc[radius=.1, start angle=180, end angle=90]--(1.4,.5) arc[radius=.1, start angle=90, end angle=0];
\draw [fill=white] (.4,.4) arc[radius=.1, start angle=180, end angle=90]--(.7,.5) arc[radius=.1, start angle=90, end angle=0];
\draw [fill=white] (0,.05) rectangle (.5,.45);
\node at (.25,.25) {$x^{\frac{1}{2}}$};
\node at (-.04,.25) {\tiny $\$$};
\draw [fill=yellow] (.7,.05) rectangle (1.2,.45);
\node at (.66,.25) {\tiny $\$$};
\node at (.95,.25) {$y^{\frac{1}{2}}$};
\draw [fill=cyan] (1.4,.05) rectangle (1.9,.45);
\node at (1.94,.25) {\tiny $\$$};
\node at (1.65,.25) {$z^{\frac{1}{2}}$};
\end{scope}
\draw [fill=lightgray]  (.1,.05) rectangle (.4,-.5);
\draw [fill=lightgray]  (.8,.05) rectangle (1.1,-.5);
\draw [fill=lightgray]  (1.5,.05) rectangle (1.8,-.5);
\begin{scope}
\begin{scope}
\begin{scope}[yscale=-1]
\draw [fill=lightgray] (.1,.4)--(.1,.5) arc[radius=.1, start angle=180, end angle=90]--(1.7,.6) arc[radius=.1, start angle=90, end angle=0]--(1.8,.4);;
\draw [fill=white] (1,.4) arc[radius=.1, start angle=180, end angle=90]--(1.4,.5) arc[radius=.1, start angle=90, end angle=0];
\draw [fill=white] (.4,.4) arc[radius=.1, start angle=180, end angle=90]--(.7,.5) arc[radius=.1, start angle=90, end angle=0];
\draw [fill=white] (0,.05) rectangle (.5,.45);
\node at (.25,.25) {$x^{\frac{1}{2}}$};
\node at (-.04,.25) {\tiny $\$$};
\draw [fill=yellow] (.7,.05) rectangle (1.2,.45);
\node at (.66,.25) {\tiny $\$$};
\node at (.95,.25) {$y^{\frac{1}{2}}$};
\draw [fill=cyan] (1.4,.05) rectangle (1.9,.45);
\node at (1.94,.25) {\tiny $\$$};
\node at (1.65,.25) {$z^{\frac{1}{2}}$};
\end{scope}
\end{scope}
\end{scope}
\end{tikzpicture}}\;\geq\;0.\label{equ: Equivalent of Schur Product}
\end{align}
Here we infer positivity from the reflection positivity in \eqref{equ: relative positivity}.
This proof works on $\cB$ by switching the shading, since the dual of a subfactor is also a subfactor.
\end{proof}
\vskip -.3cm
The Schur product property holds on both $\cA$ and $\cB$ for subfactors. On fusion rings, it holds on $\cA$ but not necessarily on $\cB$. We discuss this essential difference revealed by QFA in more detail in \S \ref{Sec: QFA on Fusion Category}.

\subsection{\bf Applications}
An important application of the Schur product theorem comes in the classification of \textit{abelian} subfactor planar algebras~\cite{Liu16}. One can regard this as 
the fundamental theorem for abelian subfactors, extending the fundamental theorem for finite abelian groups. This was the first 
classification that neither requires a bound on the Jones index, nor on the dimension. 
\begin{theorem}[Theorem 6.7~\cite{Liu16}] 
An irreducible subfactor planar algebra is abelian if and only if it is a free product of the simplest  (Temperley-Lieb-Jones) planar algebras, and finite abelian groups.
\end{theorem}

Also one obtains 
a geometric proof of reflection positivity. 
\begin{theorem}
Reflection positivity holds for Hamiltonians on planar para algebras (Theorem {7.1} of~\cite{JL17}) and on Levin-Wen models (Theorem {3.2} of~\cite{JL19}).
\end{theorem}

\subsection{Quantum Hausdorff-Young Inequalities}
Many estimates for the norm $K_{p,q}=\norm{\mathfrak{F}}_{L^p\rightarrow L^q}$ have been established in the quantum case, including when $p,q$ are not dual, see~\cite{JLW16,LW19}. Instead of synthesizing these estimates into one theorem, we give a picture~(\ref{Pic:subfactor-bounds}) illustrating the known bounds for $K_{p,q}$, with the extremizers for various regions and boundaries in text. 
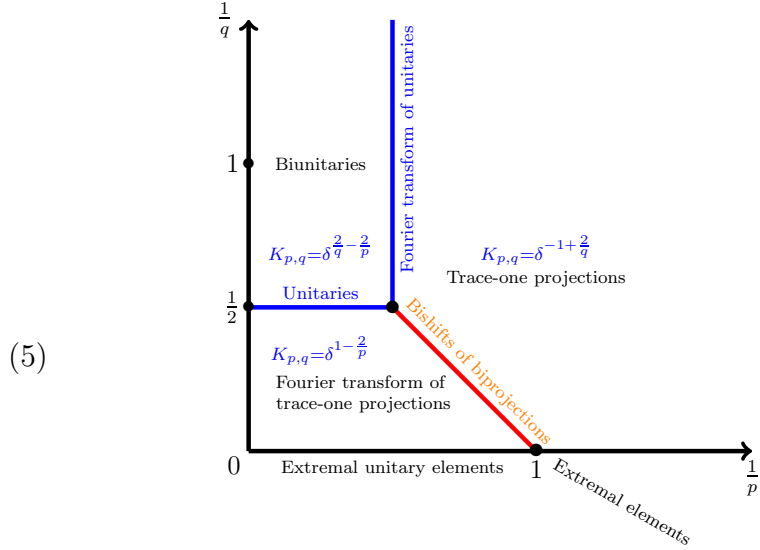
\begin{equation}\label{Pic:subfactor-bounds}
\scalebox{.85}{
\raisebox{-3cm}{\begin{tikzpicture}[scale=2.25]
\draw [line width=2] [->](0,0)--(0,3);
\node at (-0.1,-0.1) {0};
\node [left] at (-.05, 3) {$\frac{1}{q}$};
\draw [line width=2]  [->](0,0)--(3.5,0);
\node [below] at (2, 0) {1};
\node [below] at (3.5,0) {$\frac{1}{p}$};
\draw [blue, line width=2] (1,1)--(0,1);
\node [rotate=-45, color=orange] at (1.6,0.55) {\tiny Bishifts of biprojections};
\node [left] at (0,1) {$\frac{1}{2}$};
\node at (0,1) {$\bullet$};
\node [left] at (0, 2) {1};
\node at (0,2) {$\bullet$};
\draw [red, line width=2] (1,1)--(2,0);
\draw [blue, line width=2] (1,1)--(1,3);
\node at (2,1.2) {\tiny Trace-one projections};
\node [blue] at (2, 1.4) {${\scriptstyle K_{p,q}=\delta^{-1+\frac{2}{q}}}$};
\node at (1, 1) {\large $\bullet$};
\node [color=blue,rotate=90] at (1.1, 2.1) { \tiny Fourier transform of unitaries};
\node [blue] at (0.5,1.4) {$\scriptstyle K_{p,q}=\delta^{\frac{2}{q}-\frac{2}{p}}$};
\node [blue] at (0.5, 1.1) {\tiny Unitaries};
\node [blue] at (.5, 0.7) {$\scriptstyle K_{p,q}=\delta^{1-\frac{2}{p}}$};
\node (example-align) [align=center] at (0.8, 0.4) {$\text{\tiny Fourier transform of }\atop
\text{\tiny trace-one projections}$};
\node [color=black] at (0.5, 2) {\tiny Biunitaries};
\node at (2,0) {\large $\bullet$};
\node [rotate=-30] at (2.6,-.35) {\tiny Extremal elements};
\node [below] at (1,0) {\tiny Extremal unitary elements};
\end{tikzpicture}}}
\end{equation}
The constant $\delta$ is the square root of the Jones index. The extremizers of these inequalities have nine different characterizations. In particular, the red line $1/p+1/q=1$,  for $1/2\leq1/p\leq1$ corresponds to the quantum Hausdorff-Young inequality. 
Moreover, all the other quantum inequalities, such as quantum Young's inequality, in Theorem \ref{Thm:Fusion} have been proved for subfactors planar algebras in \cite{JLW16}.

\section{QFA on Unitary Fusion Categories}\label{Sec: QFA on Fusion Category}.
The quantum Fourier analysis on subfactors also works for unitary fusion categories through the quantum double construction, see e.g. \cite{Muger03}. Let $\cC$ be a unitary fusion category and $I=\{X_1,X_2,\cdots,X_m\}$ be the set of simple objects. There is a Frobenius algebra $\gamma$ in $\cC\boxtimes\overline{\cC}$ whose object is  $\displaystyle\bigoplus_{i=1}^m X_i\boxtimes \overline{X_i}$.
Following the quantum double construction, we obtain an irreducible subfactor planar algebra, such that $\cA=\mathscr{P}_{2,+}=\hom_{\cC\boxtimes\overline{\cC}}(\gamma)$ and $\cB=\mathscr{P}_{2,-}=\hom_{\gamma-\gamma}(\gamma\otimes \gamma)$. 
 Applying quantum Fourier analysis to $\cA$ on this subfactor, we obtain inequalities on the Grothendieck ring of unitary fusion categories as stated in Theorem \ref{Thm:Fusion}. Applying quantum Fourier analysis to $\cB$, we obtain inequalities on the dual of the Grothendieck ring, which turn out to be highly non-trivial.

\subsubsection{\bf Application: Analytic Obstructions}\label{Sect:ApplicationAnalytic}
It is important to determine whether a fusion ring can be the Grothendieck ring of a unitary fusion category. Quantum Fourier analysis provides powerful analytic obstructions to the unitary categorification of fusion rings. The quantum inequalities in Theorem \ref{Thm:Fusion} holds on the dual of Grothendieck rings. However, they may not necessarily hold on the dual of fusion rings, thereby providing analytic obstructions of the unitary categorification of fusion rings. 

The Schur product property on the dual of a fusion ring, $0\leq x*y$ if $0\leq x,y\in \cB$, is a surprisingly efficient analytic obstruction of the unitary categorification: 
\begin{theorem}[\cite{LPW19}]\label{thm: Schur Obs}
If a fusion ring can be unitarily categorified, then the Schur product property holds on the dual. 
\end{theorem}


There are 21 exmaples in the classification of simple integral fusion rings up to rank 8 and global dimension 989. Four of them are group-like. Methods based on previously known analytic, algebraic and number theoretic obstructions did not determine whether the remaining 17 could be unitarily categorified. Due to Schur product obstruction, 15 out of the 17 have no unitary categorification, as shown in \cite{LPW19}.

\noindent {\bf Example:}
Let us recall one example from  \cite{LPW19} to illustrate this obstruction.  Let $\fA$ be the rank-7 simple integral fusion ring with the following seven  fusion matrices, 
equal to  
\[ 
\scalebox{.85}{$
\left(\begin{smallmatrix}
1 & 0 & 0 & 0& 0& 0& 0 \\
0 & 1 & 0 & 0& 0& 0& 0 \\
0 & 0 & 1 & 0& 0& 0& 0 \\
0 & 0 & 0 & 1& 0& 0& 0 \\
0 & 0 & 0 & 0& 1& 0& 0 \\
0 & 0 & 0 & 0& 0& 1& 0 \\
0 & 0 & 0 & 0& 0& 0& 1 
\end{smallmatrix}\right)\;,
\
\left(
\begin{smallmatrix}
0 & 1 & 0 & 0& 0& 0& 0 \\
1 & 1 & 0 & 1& 0& 1& 1 \\
0 & 0 & 1 & 0& 1& 1& 1 \\
0 & 1 & 0 & 0& 1& 1& 1 \\
0 & 0 & 1 & 1& 1& 1& 1 \\
0 & 1 & 1 & 1& 1& 1& 1 \\
0 & 1 & 1 & 1& 1& 1& 1 
\end{smallmatrix}\right) \;,
\
\left(\begin{smallmatrix}
0 & 0 & 1 & 0& 0& 0& 0 \\
0 & 0 & 1 & 0& 1& 1& 1 \\
1 & 1 & 1 & 0& 0& 1& 1 \\
0 & 0 & 0 & 1& 1& 1& 1 \\
0 & 1 & 0 & 1& 1& 1& 1 \\
0 & 1 & 1 & 1& 1& 1& 1 \\
0 & 1 & 1 & 1& 1& 1& 1 
\end{smallmatrix}\right) \;,
\
\left(\begin{smallmatrix}
0 & 0 & 0 & 1& 0& 0& 0 \\
0 & 1 & 0 & 0& 1& 1& 1 \\
0 & 0 & 0 & 1& 1& 1& 1 \\
1 & 0 & 1 & 1& 0& 1& 1 \\
0 & 1 & 1 & 0& 1& 1& 1 \\
0 & 1 & 1 & 1& 1& 1& 1 \\
0 & 1 & 1 & 1& 1& 1& 1 
\end{smallmatrix}\right)\;,
$}
\]
\[
\scalebox{.85}{$
\left(
\begin{smallmatrix}
0 & 0 & 0 & 0& 1& 0& 0 \\
0 & 0 & 1 & 1& 1& 1& 1 \\
0 & 1 & 0 & 1& 1& 1& 1 \\
0 & 1 & 1 & 0& 1& 1& 1 \\
1 & 1 & 1 & 1& 1& 1& 1 \\
0 & 1 & 1 & 1& 1& 2& 1 \\
0 & 1 & 1 & 1& 1& 1& 2 
\end{smallmatrix} \right) \;,
\
\left(
\begin{smallmatrix}
0 & 0 & 0 & 0& 0& 1& 0 \\
0 & 1 & 1 & 1& 1& 1& 1 \\
0 & 1 & 1 & 1& 1& 1& 1 \\
0 & 1 & 1 & 1& 1& 1& 1 \\
0 & 1 & 1 & 1& 1& 2& 1 \\
1 & 1 & 1 & 1& 2& 0& 3 \\
0 & 1 & 1 & 1& 1& 3& 1 
\end{smallmatrix} \right) \;,
\
\left(
\begin{smallmatrix}
0 & 0 & 0 & 0& 0& 0& 1 \\
0 & 1 & 1 & 1& 1& 1& 1 \\
0 & 1 & 1 & 1& 1& 1& 1 \\
0 & 1 & 1 & 1& 1& 1& 1 \\
0 & 1 & 1 & 1& 1& 1& 2 \\
0 & 1 & 1 & 1& 1& 3& 1 \\
1 & 1 & 1 & 1& 2& 1& 2 
\end{smallmatrix}\right)\;.
$}
\]
The eigenvalue table of these matrices (where $\zeta^{7}=1$) is:  
\begin{equation}\label{Characters}   
\scalebox{.9}{$\begin{pmatrix}  
1 & 1 & 1 & 1 & 1 & 1 & 1  \\
5 & -1 & -\zeta -\zeta^6 & -\zeta^5 - \zeta^2 & -\zeta^4 - \zeta^3 & 0 & 0 \\ 
5 & -1 & -\zeta^5 - \zeta^2 & -\zeta^4 - \zeta^3 & -\zeta -\zeta^6 & 0 & 0  \\ 
5 & -1 & -\zeta^4 - \zeta^3 & -\zeta -\zeta^6 & -\zeta^5 - \zeta^2 & 0 & 0  \\ 
6 & 0 & -1 & -1 & -1 & 1 & 1  \\
7 & 1 & 0 & 0 & 0 & 0 & -3  \\
7 & 1 & 0 & 0 & 0 & -1 & 2
\end{pmatrix}$}\;.
\end{equation}
The first column is the Perron-Frobenius dimension of the 7 simple objects.
Take $X=x_1+x_5-3x_6+2x_7$, then $X=X^*=X^2/15$.

The Schur product property on $\cB$, equivalent to a dual version of \eqref{equ: Equivalent of Schur Product}, yields (with $x=y=z=X$) that
\begin{equation}\label{Obstruction}
 d((\widehat{XX^*})\diamond (\widehat{XX^*})\diamond (\widehat{XX^*}))\geq 0.
\end{equation}
However, it follows directly that \eqref{Obstruction} is false in this case, as 
$$ \frac{1^3}{1} + \frac{0^3}{5} + \frac{0^3}{5} + \frac{0^3}{5} + \frac{1^3}{6} + \frac{(-3)^3}{7} + \frac{2^3}{7} =  -\frac{65}{42}<0.$$
Therefore, the fusion ring $\fA$ can not be unitarily categorified. 

\section{QFA on Locally Compact Quantum Groups}\label{sec: QFA on LCQG}
The previous results focus on finite quantum symmetry, such as fusion rings and finite-index subfactors. One might ask whether quantum Fourier analysis can be established for infinite quantum symmetry. The answer is ``yes;'' there are results on infinite-dimensional Kac algebras and locally compact quantum groups.

We recall the definition of the Fourier transform on locally compact quantum groups, of which the Fourier transform on Kac algebra is a special case, see \cite{KuVaes}.
Let $\mathbb{G}$ be a locally compact quantum group and $\varphi$ the left Haar weight.
Suppose $W$ is the multiplicative unitary,  $\phi$ is a normal semi-finite faithful weight on the commutant of $L^\infty(\mathbb{G})$, $\hat{\phi}$ is a normal semi-finite faithful weight on the commutant of $L^\infty(\widehat{\mathbb{G}})$.
Let $d=\frac{d\varphi}{d\phi}$, $\hat{d}=\frac{d\hat{\varphi}}{d\hat{\phi}}$ be the Connes' spatial derivatives,  $L^p(\phi), L^p(\hat{\phi})$ Hilsum's space for any $1\leq p\leq\infty$.
The Fourier transform $\mathcal{F}_p: L^p(\phi)\to L^q(\hat{\phi})$, for $1\leq p\leq 2$, and $1/p+1/q=1$, is defined by
$$\mathcal{F}_p(xd^{1/p})=(\varphi\otimes\iota)(W(x\otimes 1))\hat{d}^{1/q}, \forall x\in \mathcal{T}_\varphi^2\;.$$
Here $\mathcal{T}_\varphi\subset \mathfrak{N}_\varphi \cap \mathfrak{N}_\varphi^*$ is the space of elements  analytic  with respect to $\varphi$.
Even the definition of the convolution on locally compact quantum groups is non-trivial.

The quantum inequalities in Theorem \ref{Thm:Fusion} on these infinite quantum symmetries have been partially studied in \cite{Coo10, Cas13, LW17,LWW, JLW18}. The quantum uncertainty principle $QUP$-$2$ in Theorem \ref{Thm:Fusion} becomes a continuous family of inequalities on locally compact quantum groups \cite{JLW18}.

\section{Surface Algebras and A Universal Inequality}
\subsection{Surface Algebras} 
Many inequalities in classical Fourier analysis have not been axiomatized in a pictorial framework. The third author introduced surface algebras in \cite{Liu19}, formalizing the extension of planar algebras from 2D to 3D space, outlined in~\cite{LWJ17}. Surface algebras are an extensive framework to capture additional pictorial features of Fourier analysis.

For any subfactor planar algebra, the actions of planar tangles can be further extended to the actions of surface tangles.  
(The arrow in planar diagrams corresponds to the $\$$ sign in planar algebras. 
The clockwise/anticlockwise orientation of the arrow indicates the input/output disc in surface algebras.)
One can represent Fourier transform,  multiplication and convolution on
as the action of the following surface tangles in the 3D space:
\begin{equation}
\raisebox{-.5cm}{
	\begin{tikzpicture}
	\begin{scope}[scale=.5]
	
	\fill[thick,blue!20] (0,0+1)--(1,0+1)--(1.5,.5+1)--(.5,.5+1)--(0,0+1);
	\draw[thick,blue] (0,0+1)--(1,0+1)--(1.5,.5+1)--(.5,.5+1)--(0,0+1);
	\draw[->,blue] (.5,1.5)--(.25,1.25);
	
	\fill[thick,blue!20] (0,0)--(1,0)--(1.5,.5)--(.5,.5)--(0,0);
	\draw[thick,blue] (0,0)--(1,0)--(1.5,.5);
	\draw[thick,blue,dashed] (1.5,.5)--(.5,.5)--(0,0);
	\draw[->,blue] (0,0)--(.5,0);
	
	\draw (0,0)--++(0,1);
	\draw (1,0)--++(0,1);
	\draw[dashed] (.5,.5)--++(0,1);
	\draw (1.5,.5)--++(0,1);
	
	\end{scope}
	\end{tikzpicture}
}
\; ,
\raisebox{-.5cm}{
	\begin{tikzpicture}
	\begin{scope}[scale=.5]
	
	\begin{scope}[shift={(0,0)}]
	\fill[thick,blue!20] (0,0)--(1,0)--(1.5,.5)--(.5,.5)--(0,0);
	\draw[thick,blue] (0,0)--(1,0)--(1.5,.5)--(.5,.5)--(0,0);
	\draw[->,blue] (.5,.5)--(.25,.25);
	\end{scope}

	\begin{scope}[shift={(1,1)}]
	\fill[thick,blue!20] (0,0)--(1,0)--(1.5,.5)--(.5,.5)--(0,0);
	\draw[thick,blue] (0,0)--(1,0)--(1.5,.5)--(.5,.5)--(0,0);
	\draw[->,blue] (.5,.5)--(.25,.25);
	\end{scope}
	
	\begin{scope}[shift={(.5,-1)}]
	\fill[thick,blue!20] (0,0)--(1,0)--(1.5,.5)--(.5,.5)--(0,0);
	\draw[thick,blue] (0,0)--(1,0)--(1.5,.5);
	\draw[thick,blue,dashed] (1.5,.5)--(.5,.5)--(0,0);
	\draw[->,blue] (.5,.5)--(.25,.25);
	\end{scope}
	
	\draw (0,0)--++(.5,-1);
	\draw (1,0)--++(.5,-1);
	\draw[dashed] (1,-.5)--(1.5,1.5);
	\draw (2,-.5)--(2.5,1.5);
	\draw (.5,.5) to [bend right=30] (1,1);
	\draw (1.5,.5) to [bend right=30] (2,1);

	\end{scope}
	\end{tikzpicture}
}
\; ,
\raisebox{-.5cm}{
	\begin{tikzpicture}
	\begin{scope}[scale=.5]
	
	\begin{scope}[shift={(0,1.5)}]
	\fill[thick,blue!20] (0,0)--(1,0)--(1.5,.5)--(.5,.5)--(0,0);
	\draw[thick,blue] (0,0)--(1,0)--(1.5,.5)--(.5,.5)--(0,0);
	\draw[->,blue] (.5,.5)--(.25,.25);
	\end{scope}

	\begin{scope}[shift={(2,1.5)}]
	\fill[thick,blue!20] (0,0)--(1,0)--(1.5,.5)--(.5,.5)--(0,0);
	\draw[thick,blue] (0,0)--(1,0)--(1.5,.5)--(.5,.5)--(0,0);
	\draw[->,blue] (.5,.5)--(.25,.25);
	\end{scope}
	
	\begin{scope}[shift={(1,0)}]
	\fill[thick,blue!20] (0,0)--(1,0)--(1.5,.5)--(.5,.5)--(0,0);
	\draw[thick,blue] (0,0)--(1,0)--(1.5,.5);
	\draw[thick,blue,dashed] (1.5,.5)--(.5,.5)--(0,0);
	\draw[->,blue] (.5,.5)--(.25,.25);
	\end{scope}
	
	\draw (1,0)--++(-1,1.5);
	\draw (2,0)--++(1,1.5);
	\draw[dashed] (1.5,.5)--++(-1,1.5);
	\draw (2.5,.5)--++(1,1.5);
	
	\draw (1,1.5) arc (-180:0:.5 and .25);
	\draw (1.5,2) arc (-180:0:.5 and .25);
	
	\end{scope}
	\end{tikzpicture}
}
\; .
\end{equation}
Using 3D pictures, one can consider the Fourier duality for surface tangles with multiple inputs and outputs, see applications in \cite{Liu19}.


One can consider 
a finite-dimensional 
Kac algebra $K$ as $\cA$ and its dual $K^*$ as $\cB$,
with a Fourier transform from $K$ to $K^*$ defined analogously to \S \ref{sec: QFA on LCQG}
The pair of Kac algebras $K$ and $K^*$ can be understood as $\cA$ and $\cB$ for the surface algegra. The co-multiplication 
is given by the following surface tangle: 
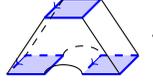
\begin{center}
	\raisebox{-.5cm}{
		\begin{tikzpicture}
		\begin{scope}[scale=.5]
		
		\begin{scope}[shift={(0,-1.5)}]
		\fill[thick,blue!20] (0,0)--(1,0)--(1.5,.5)--(.5,.5)--(0,0);
		\draw[thick,blue] (0,0)--(1,0)--(1.5,.5);
		\draw[thick,blue,dashed] (1.5,.5)--(.5,.5)--(0,0);
		\draw[->,blue] (.5,.5)--(.25,.25);
		\end{scope}

		\begin{scope}[shift={(2,-1.5)}]
		\fill[thick,blue!20] (0,0)--(1,0)--(1.5,.5)--(.5,.5)--(0,0);
		\draw[thick,blue] (0,0)--(1,0)--(1.5,.5);
		\draw[thick,blue,dashed] (1.5,.5)--(.5,.5)--(0,0);
		\draw[->,blue] (.5,.5)--(.25,.25);
		\end{scope}
		
		\begin{scope}[shift={(1,0)}]
		\fill[thick,blue!20] (0,0)--(1,0)--(1.5,.5)--(.5,.5)--(0,0);
		\draw[thick,blue] (0,0)--(1,0)--(1.5,.5)--(.5,.5)--(0,0);
		\draw[->,blue] (.5,.5)--(.25,.25);
		\end{scope}
		
		\draw (1,0)--++(-1,-1.5);
		\draw (2,0)--++(1,-1.5);
		\draw[dashed] (1.5,.5)--++(-1,-1.5);
		\draw (2.5,.5)--++(1,-1.5);
		
		\draw (1,-1.5) arc (180:0:.5 and .25);
		\draw[dashed] (1.5,-1) arc (180:0:.5 and .25);
		
		\end{scope}
		\end{tikzpicture}
	}.
\end{center}
The Hopf-axiom that the co-multiplication is an algebraic homomorphism reduces to the \emph{string-genus relation} of surface tangles~\cite{LWJ17}.

\subsection{A Universal Inequality} In a subfactor planar/surface algebra $\sP$, the Fourier transform, the multiplication, and the convolution can be realized by planar/surface tangles. 
In general, a surface tangle is a multi-linear map on $\bigoplus_{n\in \mathbb{N}} \sP_{n,\pm}$. Now we give a new pictorial inequality in the quantum case, motivated by the classical Brascamp-Lieb inequality. We replace the dual of the linear map $B_j: \mathbb{R}^n\to \mathbb{R}^{k_j}$ by a surface tangle $T_j$ with $k_j$ input discs and $n$ output discs, moreover the $n$-output discs are identical for different $j$: 
\begin{equation}\label{eq:top-braslieb}
\left\| \prod_{j=1}^m T_j (x_j) \right\|_1\leq C \prod_{j=1}^m \|x_j\|_{p_j}\;,
\end{equation}
and $C$ is the best constant.

This topological inequality includes the quantum Hausdorff-Young inequality, quantum H\"{o}lder inequality and quantum version of Young's inequality. The best constants of these three inequalities are achieved at biprojections. 

For those familiar with the Quon language \cite{LWJ17}, we can consider the new pictorial inequalities whose $T_j$'s are surface tangles with braided charged strings. In particular, if all the inputs and outputs are 2-boxes, corresponding to qudits, then $T_j$ can be any Clifford transformation on qudits. These Clifford transformations can be considered as a quantum analogue of the dual of a linear transformation $B_j: (\mathbb{Z}_d)^n \to (\mathbb{Z}_d)^{k_j}$. Considering the action on density matrices, the $n$-qudit Clifford gates on Pauli matrices are symplectic transformations on $2n$-dimensional symplectic spaces over $\mathbb{Z}_d$.

\section{Relative Inequalities, Entropy, and  Uncertainty}
Here we present a relative, quantum, Hausdorff-Young inequality. This  leads to  a new,  relative, quantum, entropic uncertainty principle. 

Let $\mathscr{P}_\bullet$ be an irreducible subfactor planar algebra with the Markov trace $\tr$. Let $\varphi$ (resp. $\psi$) be a faithful state on $\sP_{2,+}$ (resp. $\sP_{2,-}$). Let $D_\varphi$ (resp. $D_{\psi}$) be the density operator of $\varphi $ (resp. $\psi$), namely, $\varphi(\cdot)=\tr(D_{\varphi}\,\cdot)$.  Note that 
\[
\raisebox{-.35cm}{
	\begin{tikzpicture}
	\path [fill=lightgray](.1,.7) rectangle (.6,-.2);
	\draw (.1,.7)--(.1,-.2);
	\draw (.6,.7)--(.6,-.2);
	\draw [fill=pink] (-.1,0) rectangle (.8,.5);
	\node at (-.2+.003,.25) {\tiny $\$$};
	\node at (.35,.25) {\tiny$xD_\varphi^{\frac{1}{p}}$};
	\end{tikzpicture}
}=
\raisebox{-.7cm}{
	$\begin{tikzpicture}
	\path [fill=lightgray] (.1,.5) rectangle (.4,2.1);
	\draw (.1,.5)--(.1,2.1);
	\draw (.4,.5)--(.4,2.1);
	\draw[fill=green] (0,.7) rectangle (.6,1.2);
	\draw [fill=yellow](-.2,1.35) rectangle (.7,1.95);
	\node at (.25,1.65) {\tiny $D_\varphi^{\frac{1}{p}-\frac{1}{2}}$};
	\node at (.3,.95) {\tiny $x D_\varphi^{\frac{1}{2}}$};
	\node at (-.1+.003,.95) {\tiny $\$$};
	\node at (-.25+.003,1.65) {\tiny $\$$};
	\end{tikzpicture}$}.
\]
Now we define a Fourier transform $\fF_{p,\varphi,{\psi}}:L^p(\sP_{2,+},\tr)\rightarrow L^{q}(\sP_{2,-},\tr)$ for $1\leq p\leq 2$, $q=p/(p-1)$ as
\begin{equation}
\fF_{p,\varphi,\psi}(xD_\varphi^{1/p})=\fF(xD_\varphi^{1/2})D_{\psi}^{1/q-1/2}\;.
\end{equation}
This Fourier transform is represented pictorially as follows,
\begin{equation}
\fF_{p,\varphi,\psi}:
\raisebox{-.7cm}{
	$\begin{tikzpicture}
	\path [fill=lightgray] (.1,.5) rectangle (.4,2.1);
	\draw (.1,.5)--(.1,2.1);
	\draw (.4,.5)--(.4,2.1);
	\draw[fill=green] (0,.7) rectangle (.6,1.2);
	\draw [fill=yellow](-.2,1.35) rectangle (.7,1.95);
	\node at (.25,1.65) {\tiny $D_\varphi^{\frac{1}{p}-\frac{1}{2}}$};
	\node at (.3,.95) {\tiny $x D_\varphi^{\frac{1}{2}}$};
	\node at (-.1+.003,.95) {\tiny $\$$};
	\node at (-.25+.003,1.65) {\tiny $\$$};
	\end{tikzpicture}$}
\quad\longmapsto\quad
\raisebox{-.8cm}{
$\begin{tikzpicture}
\path [fill=lightgray] (-.6,.4) rectangle (.9,2.2);
\draw [fill=white] (.1,2.2)--(.1,.7) arc[radius=.15, start angle=360, end angle=180]--(-.2,2.2);
\draw [fill=white] (.4,.4)--(.4,1.2) arc[radius=.15, start angle=180, end angle=0]--(.7,.4);
\draw[fill=green] (0,.7) rectangle (.6,1.2);
\draw [fill=cyan](-.2-.25,1.4) rectangle (.7-.25,2);
\node at (0,1.7) {\tiny $D_\psi^{\frac{1}{q}-\frac{1}{2}}$};
\node at (.3,.95) {\tiny $xD_\varphi^{\frac{1}{2}}$};
\node at (-.1+.003,.95) {\tiny $\$$};
\node at (-.5+.003,1.7) {\tiny $\$$};
\end{tikzpicture}$}\;.
\end{equation}

From Plancherel's theorem for $\fF$, we infer Plancherel's theorem for $\fF_{p,\varphi,\psi}$,
\begin{equation}
\|\fF_{2,\varphi,{\psi}}(xD_\varphi^{1/2})\|_2=\|\fF(xD_\varphi^{1/2})\|_2=\|xD_\varphi^{1/2}\|_2.
\end{equation}

\begin{theorem}[\bf Relative, quantum Hausdorff-{Young} inequality]\label{thm: relative HY}
	Let $\sP_\bullet$ be an irreducible subfactor planar algebra and let  $\varphi,\psi$ be faithful states on $\sP_{2,\pm}$. Then for any $x\in \sP_{2,+}, 1\leq p\leq 2$,  and dual  $2\leq q=p/(p-1)$, we have 
	\begin{equation}\label{equ: RQHY}
	\|\fF_{p,\varphi,{\psi}}(xD_\varphi^{1/p})\|_{q}\leq K_{p,\varphi,\psi}\|xD_\varphi^{1/p}\|_p\;.
	\end{equation}
Here $K_{p,\varphi,\psi}=\delta^{-2/p}\|D_{\psi}^{1/q-1/2}\|_\infty \|\fF(D_\varphi^{1/2-1/p})\|_1$. Pictorially, 
\[
\left\|\scalebox{.75}{$
\raisebox{-.8cm}{
$\begin{tikzpicture}
\path [fill=lightgray] (-.6,.4) rectangle (.9,2.2);
\draw [fill=white] (.1,2.2)--(.1,.7) arc[radius=.15, start angle=360, end angle=180]--(-.2,2.2);
\draw [fill=white] (.4,.4)--(.4,1.2) arc[radius=.15, start angle=180, end angle=0]--(.7,.4);
\draw[fill=green] (0,.7) rectangle (.6,1.2);
\draw [fill=cyan](-.2-.25,1.4) rectangle (.7-.25,2);
\node at (0,1.7) {\tiny $D_\psi^{\frac{1}{q}-\frac{1}{2}}$};
\node at (.3,.95) {\tiny $xD_\varphi^{\frac{1}{2}}$};
\node at (-.1+.003,.95) {\tiny $\$$};
\node at (-.5+.003,1.7) {\tiny $\$$};
\end{tikzpicture}$}
	$}\,
\right\|_{q}
\leq 
\left(
\left\Vert\,\raisebox{-.45cm}{\begin{tikzpicture}
\draw (0,-.5)--(0,.55);
\end{tikzpicture}}\,
\right\Vert^{-2/p}_{1}
  \left\Vert\,\raisebox{-.35cm}{
  \begin{tikzpicture}
  \path [fill=lightgray] (-.5,-.2) rectangle (.9,.7);
  \path [fill=white](.1,.7) rectangle (.6,-.2);
  \draw (.1,.7)--(.1,-.2);
  \draw (.6,.7)--(.6,-.2);
  \draw [fill=cyan] (-.2,0) rectangle (.8,.5);
  \node at (-.3+.003,.25) {\tiny $\$$};
  \node at (.35,.25) {\tiny$D_\psi^{\frac{1}{q}-\frac{1}{2}}$};
  \end{tikzpicture}}\,
\right\Vert_{\infty}\,\,
  \left\Vert\,\raisebox{-.35cm}{
  \begin{tikzpicture}
  \path [fill=lightgray] (-.3,-.2) rectangle (1,.7);
  \draw [fill=white](.1,.7)--(.1,0) arc[radius=.15, start angle=360, end angle=180]--(-.2,.7);
  \draw [fill=white](.6,-.2)--(.6,.5) arc[radius=.15, start angle=180, end angle=0]--(.9,-.2);
  \draw [fill=yellow] (-.05,0) rectangle (.75,.5);
  \node at (-.1+.003,.25) {\tiny $\$$};
  \node at (.35,.25) {\tiny$D_\varphi^{\frac{1}{2}-\frac{1}{p}}$};
  \end{tikzpicture}}\,
\right\Vert_{1}
\right)
\norm{\hskip -.17cm \raisebox{-.35cm}{
	\begin{tikzpicture}
	\path [fill=lightgray](.1,.7) rectangle (.6,-.2);
	\draw (.1,.7)--(.1,-.2);
	\draw (.6,.7)--(.6,-.2);
	\draw [fill=pink] (-.1,0) rectangle (.8,.5);
	\node at (-.2+.003,.25) {\tiny $\$$};
	\node at (.35,.25) {\tiny$xD_\varphi^{\frac{1}{p}}$};
	\end{tikzpicture}
}}_p\;.
\]
\end{theorem}

\begin{proof}  We give the elementary and insightful picture proof:
	\begin{align*}
	&\hskip-.5cm\left\|\scalebox{.75}{$
\raisebox{-.8cm}{
$\begin{tikzpicture}
\path [fill=lightgray] (-.6,.4) rectangle (.9,2.2);
\draw [fill=white] (.1,2.2)--(.1,.7) arc[radius=.15, start angle=360, end angle=180]--(-.2,2.2);
\draw [fill=white] (.4,.4)--(.4,1.2) arc[radius=.15, start angle=180, end angle=0]--(.7,.4);
\draw[fill=green] (0,.7) rectangle (.6,1.2);
\draw [fill=cyan](-.2-.25,1.4) rectangle (.7-.25,2);
\node at (0,1.7) {\tiny $D_\psi^{\frac{1}{q}-\frac{1}{2}}$};
\node at (.3,.95) {\tiny $xD_\varphi^{\frac{1}{2}}$};
\node at (-.1+.003,.95) {\tiny $\$$};
\node at (-.5+.003,1.7) {\tiny $\$$};
\end{tikzpicture}$}
	$}\,
\right\|_{q}
=\left\|\raisebox{-.9cm}{
\scalebox{.75}{
	\begin{tikzpicture}
	\path [fill=lightgray] (-.6,.-.4) rectangle (1.2,2.2);
	\draw [fill=white] (.2,2.2)--(.2,-.1) arc[radius=.2, start angle=360, end angle=180]--(-.2,2.2);
	\draw [fill=white] (.6,-.4)--(.6,1.1) arc[radius=.2, start angle=180, end angle=0]--(1,-.4);
	\draw[fill=yellow] (0,.6) rectangle (.8,1.1);
	\draw[fill=pink] (.1,-.1) rectangle (.7,.4);
	\draw [fill=cyan](-.2-.25,1.4) rectangle (.7-.25,2);
	\node at (.3,.15) {\tiny $xD_\varphi^{1/p}$};
	\node at (0,1.7) {\tiny $D_\psi^{\frac{1}{q}-\frac{1}{2}}$};
	\node at (.4,.85) {\tiny $D_\varphi^{\frac{1}{2}-\frac{1}{p}}$};
	\node at (-.1+.003,.95) {\tiny $\$$};
	\node at (-.5+.003,1.7) {\tiny $\$$};
	\end{tikzpicture}
	}
}\right\|_q=
\left\|\scalebox{.75}{$
\raisebox{-1.05cm}{$
	\begin{tikzpicture}
	\path [fill=lightgray] (-.6,.1) rectangle (2.4,2.4);
	\draw [fill=white] (.1,2.4)--(.1,1.7) arc[radius=.1, start angle=180, end angle=270]--(1.35,1.6) arc[radius=.1, start angle=90, end angle=0]--(1.45,.7) arc[radius=.15, start angle=0, end angle=-90]--(1.15,.55) arc[radius=.15, start angle=270, end angle=180]--(1,1.35) arc[radius=.1, start angle=0, end angle=90]--(.2,1.45) arc[radius=.1, start angle=90, end angle=180]--(.1,.7) arc[radius=.15, start angle=360, end angle=180]--(-.2,2.4);
	\draw [fill=white] (.4,0.1)--(.4,1.2) arc[radius=.15, start angle=180, end angle=0]--(.7,.45)arc[radius=.1, start angle=180, end angle=270]--(1.65,.35) arc[radius=.1, start angle=270, end angle=360]--(1.75,1.2) arc[radius=.2, start angle=180, end angle=0]--(2.15,0.1);
	\draw[fill=pink] (0,.7) rectangle (.6,1.2);
	\draw [fill=yellow] (1.2,.7) rectangle (2,1.2);
	\draw [fill=cyan](-.15-.25,1.7) rectangle (.65-.25,2.2);
	\node at (0,1.95) {\tiny $D_\psi^{\frac{1}{q}-\frac{1}{2}}$};
	\node at (1.6,.95) {\tiny $D_\varphi^{\frac{1}{2}-\frac{1}{p}}$};
	\node at (.3,.95) {\tiny $xD_\varphi^{\frac{1}{p}}$};
	\node at (-.1+.003,.95) {\tiny $\$$};
	\node at (-.5+.003,1.95) {\tiny $\$$};
	\node at (1.1+.003,.95) {\tiny $\$$};
	\end{tikzpicture}
	$}
	$}
\,\right\|_{q}\\
\,=&\displaystyle\sup_{\|y\|_p=1}\left\vert\,\scalebox{.63}{$
\raisebox{-1.4cm}{$
	\begin{tikzpicture}
	\path [fill=lightgray] (-.6,-.15) rectangle (2.85,2.9);
	\path [fill=white] (.1,2.4) arc[radius=.15, start angle=180, end angle=90]--(2.3,2.55) arc[radius=.15, start angle=90, end angle=0]--(2.45,.4) arc[radius=.15, start angle=360, end angle=180]--(.4,.1)--(.4,.2) arc[radius=.15, start angle=180, end angle=270]--(2.55,.05) arc[radius=.15, start angle=270, end angle=360]--(2.7,2.6) arc[radius=.15, start angle=0, end angle=90]--(-.05,2.75) arc[radius=.15, start angle=90, end angle=180]--(-.2,2.4);
	\draw (.1,2.4) arc[radius=.15, start angle=180, end angle=90]--(2.3,2.55) arc[radius=.15, start angle=90, end angle=0]--(2.45,.4) arc[radius=.15, start angle=360, end angle=180]--(.4,.25)--(.4,.2) arc[radius=.15, start angle=180, end angle=270]--(2.55,.05) arc[radius=.15, start angle=270, end angle=360]--(2.7,2.6) arc[radius=.15, start angle=0, end angle=90]--(-.05,2.75) arc[radius=.15, start angle=90, end angle=180]--(-.2,2.4);
	\draw [fill=white] (.1,2.4)--(.1,1.7) arc[radius=.1, start angle=180, end angle=270]--(1.35,1.6) arc[radius=.1, start angle=90, end angle=0]--(1.45,.7) arc[radius=.15, start angle=0, end angle=-90]--(1.15,.55) arc[radius=.15, start angle=270, end angle=180]--(1,1.35) arc[radius=.1, start angle=0, end angle=90]--(.2,1.45) arc[radius=.1, start angle=90, end angle=180]--(.1,.7) arc[radius=.15, start angle=360, end angle=180]--(-.2,2.4);
	\draw [fill=white] (.4,0.1)--(.4,1.2) arc[radius=.15, start angle=180, end angle=0]--(.7,.45)arc[radius=.1, start angle=180, end angle=270]--(1.65,.35) arc[radius=.1, start angle=270, end angle=360]--(1.75,1.2) arc[radius=.2, start angle=180, end angle=0]--(2.15,.4);
	\draw[fill=pink] (0,.7) rectangle (.6,1.2);
	\draw [fill=yellow] (1.2,.7) rectangle (2,1.2);
	\draw [fill=cyan](-.15-.25,1.7) rectangle (.85-.25,2.2);
	\node at (0.08,1.95) {\tiny $D_\psi^{\frac{1}{q}-\frac{1}{2}}y$};
	\node at (1.65,.95) {\tiny $D_\varphi^{\frac{1}{2}-\frac{1}{p}}$};
	\node at (.35,.95) {\tiny $xD_\varphi^{\frac{1}{p}}$};
	\node at (-.1+.003,.95) {\tiny $\$$};
	\node at (-.5+.003,1.95) {\tiny $\$$};
	\node at (1.1+.003,.95) {\tiny $\$$};
	\end{tikzpicture}
	$}
$}\,\right\vert=\displaystyle\sup_{\|y\|_p=1}\left\vert\,\scalebox{.61}{$
\raisebox{-1.45cm}{$
	\begin{tikzpicture}
	\path [fill=lightgray] (-.8,-.6) rectangle (2.8,2.55);
    \draw [fill=white] (-.2,1.45)--(-.2,0) arc[radius=.15, start angle=180, end angle=360]--(.1,.7) arc[radius=.1, start angle=180, end angle=90]--(2.1,.8) arc[radius=.1, start angle=270, end angle=360]--(2.2,1.95) arc[radius=.15, start angle=0,end angle=180]--(1.9,1.2) arc[radius=.1, start angle=360, end angle=270]--(.2,1.1) arc[radius=.1, start angle=270, end angle=180]--(.1,1.45);
    \draw [fill=white](.6,.5) arc[radius=.15,, start angle=180, end angle=0]--(.9,0) arc[radius=.1, start angle=180, end angle=270]--(2.3,-.1) arc[radius=.1, start angle=270, end angle=360]--(2.4,2.1) arc[radius=.15, start angle=0, end angle=90]--(1.45,2.25) arc[radius=.15, start angle=90, end angle=180]--(1.3,1.45) arc[radius=.15, start angle=0, end angle=-180]--(1,1.95) arc[radius=.15, start angle=0, end angle=90]--(.25,2.1)arc[radius=.15, start angle=90, end angle=180]--(-.2,1.95)--(-.2,2.3) arc[radius=.1, start angle=180, end angle=90]--(2.5,2.4) arc[radius=.1, start angle=90, end angle=0]--(2.6,-.3) arc[radius=.1, start angle=0, end angle=-90]--(.75,-.4) arc[radius=.15, start angle=270, end angle=180]--(.6,0);
    \draw [fill=pink] (0,0) rectangle (.7,.5);
    \draw [fill=yellow] (1.2,1.45) rectangle (2,1.95);
	\draw [fill=cyan](-.25-.25,1.45) rectangle (.8-.25,1.95);
	\node at (0,1.7) {\tiny $D_\psi^{\frac{1}{q}-\frac{1}{2}}y$};
	\node at (1.6,1.7) {\tiny $D_\varphi^{\frac{1}{2}-\frac{1}{p}}$};
	\node at (.35,.25) {\tiny$xD_\varphi^{\frac{1}{p}}$};
	\node at (-.1+.003,.25) {\tiny $\$$};
	\node at (-.6+.003,1.7) {\tiny $\$$};
	\node at (2.1-.003,1.7) {\tiny $\$$};
	\end{tikzpicture}
	$}
$}\,\right\vert\\
\leq& \displaystyle\widetilde{K}\left\|\raisebox{-.38cm}{
\begin{tikzpicture}
\path [fill=lightgray] (-.3,-.2) rectangle (1,.7);
\draw [fill=white](.1,.7)--(.1,0) arc[radius=.15, start angle=360, end angle=180]--(-.2,.7);
\draw [fill=white](.6,-.2)--(.6,.5) arc[radius=.15, start angle=180, end angle=0]--(.9,-.2);
\draw [fill=pink] (0,0) rectangle (.7,.5);
\node at (-.1+.003,.25) {\tiny $\$$};
\node at (.35,.25) {\tiny$xD_\varphi^{\frac{1}{p}}$};
\end{tikzpicture}
}\right\|_{q}
 \leq \delta^{1-2/p}\displaystyle\widetilde{K}
	\norm{\hskip -.17cm \raisebox{-.35cm}{
	\begin{tikzpicture}
	\path [fill=lightgray](.1,.7) rectangle (.6,-.2);
	\draw (.1,.7)--(.1,-.2);
	\draw (.6,.7)--(.6,-.2);
	\draw [fill=pink] (-.1,0) rectangle (.8,.5);
	\node at (-.2+.003,.25) {\tiny $\$$};
	\node at (.35,.25) {\tiny$xD_\varphi^{\frac{1}{p}}$};
	\end{tikzpicture}
}}_p\;.
\end{align*}
The first inequality is a consequence of the quantum  H\"older inequality. The second inequality is a consequence of the quantum Hausdorff-Young inequality~(Theorem 4.8~\cite{JLW16}),  along with 
\[
\delta = 
\left\Vert\,\raisebox{-.28cm}{\begin{tikzpicture}
\path [fill=lightgray] (0,-.35) rectangle (.1,.355);
\path (0,-.35) rectangle (-.1,.355);
\draw (0,-.35)--(0,.355);
\end{tikzpicture}}\,
\right\Vert_{1}\;.
\]
Here the constant is $\widetilde{K}=\sup_{\norm{y}_p=1}\widetilde{K}(y)$, with $\widetilde{K}(y)$  equal to the following picture:
\begin{align*}
    &\norm{
    \scalebox{.65}{$
    \raisebox{-.6cm}{
	\begin{tikzpicture}
	\path [fill=lightgray] (-.8,.8) rectangle (2.5,2.2);
	\draw [fill=white](2.2,.8)--(2.2,1.95) arc[radius=.15, start angle=0,end angle=180]--(1.9,1.2) arc[radius=.1, start angle=360, end angle=270]--(.2,1.1) arc[radius=.1, start angle=270, end angle=180]--(.1,1.45)--(-.2,1.45)--(-.2,.8);
	\draw [fill=white](-.2, 2.2)--(-.2,1.95)--(.1,1.95) arc[radius=.15, start angle=180, end angle=90]--(.9,2.1) arc[radius=.1, start angle=90, end angle=0]--(1,1.45) arc[radius=.15, start angle=180, end angle=360]--(1.3,2.2);
	\draw [fill=yellow] (1.2,1.45) rectangle (2,1.95);
	\draw [fill=cyan](-.25-.25,1.45) rectangle (.8-.25,1.95);
	\node at (0,1.7) {\tiny $D_\psi^{\frac{1}{q}-\frac{1}{2}}y$};
	\node at (1.6,1.7) {\tiny $D_\varphi^{\frac{1}{2}-\frac{1}{p}}$};
	\node at (-.6+.003,1.7) {\tiny $\$$};
	\node at (2.1-.003,1.7) {\tiny $\$$};
	\end{tikzpicture}
	}$}
	}_{p}\leq \frac{1}{\delta}
	\norm{\raisebox{-.35cm}{
	\begin{tikzpicture}
	\path [fill=lightgray] (-.5,-.2) rectangle (1,.7);
	\path [fill=white](.1,.7) rectangle (.6,-.2);
	\draw (.1,.7)--(.1,-.2);
	\draw (.6,.7)--(.6,-.2);
	\draw [fill=cyan] (-.2,0) rectangle (.9,.5);
	\node at (-.3+.003,.25) {\tiny $\$$};
	\node at (.35,.25) {\tiny$D_\psi^{\frac{1}{q}-\frac{1}{2}}y$};
	\end{tikzpicture}
}
	}_{p}\,
	\norm{\raisebox{-.35cm}{
	\begin{tikzpicture}
	\path [fill=lightgray] (-.3,-.2) rectangle (1,.7);
	\draw [fill=white](.1,.7)--(.1,0) arc[radius=.15, start angle=360, end angle=180]--(-.2,.7);
	\draw [fill=white](.6,-.2)--(.6,.5) arc[radius=.15, start angle=180, end angle=0]--(.9,-.2);
	\draw [fill=yellow] (-.05,0) rectangle (.75,.5);
	\node at (-.1-.02,.25) {\tiny $\$$};
	\node at (.35,.25) {\tiny$D_\varphi^{\frac{1}{2}-\frac{1}{p}}$};
	\end{tikzpicture}
}}_1\\
&\qquad\leq \frac{1}{\delta}\norm{\raisebox{-.35cm}{
	\begin{tikzpicture}
	\path [fill=lightgray] (-.5,-.2) rectangle (.9,.7);
	\path [fill=white](.1,.7) rectangle (.6,-.2);
	\draw (.1,.7)--(.1,-.2);
	\draw (.6,.7)--(.6,-.2);
	\draw [fill=cyan] (-.2,0) rectangle (.8,.5);
	\node at (-.3+.003,.25) {\tiny $\$$};
	\node at (.35,.25) {\tiny$D_\psi^{\frac{1}{q}-\frac{1}{2}}$};
	\end{tikzpicture}
}}_{\infty}\,
	\norm{\raisebox{-.35cm}{
	\begin{tikzpicture}
	\path [fill=lightgray] (-.3,-.2) rectangle (1,.7);
	\draw [fill=white](.1,.7)--(.1,0) arc[radius=.15, start angle=360, end angle=180]--(-.2,.7);
	\draw [fill=white](.6,-.2)--(.6,.5) arc[radius=.15, start angle=180, end angle=0]--(.9,-.2);
	\draw [fill=yellow] (-.05,0) rectangle (.75,.5);
	\node at (-.1-.025,.25) {\tiny $\$$};
	\node at (.35,.25) {\tiny$D_\varphi^{\frac{1}{2}-\frac{1}{p}}$};
	\end{tikzpicture}
}}_1\norm{y}_{p}\;.
\end{align*}
To obtain the first inequality, we use the quantum Young inequality (Theorem 4.13 in~\cite{JLW16}). To obtain the second inequality, we use the quantum H\"older inequality.
\end{proof}

\subsection{Relative Entropy}
We formulate relative-entropy (RE) and the corresponding relative entropic  quantum (REQ) uncertainty principle.
For two weights $\omega, \varphi$ on $\sP_{2,+}$, recall that  the relative entropy~\cite{Araki76} is  
$$S(\omega\|\varphi)=\tr_2(D_\omega(\log D_\omega-\log D_\varphi))\;.$$

\subsection{The Relative Entropic Quantum Uncertainty Principle}
For a weight $\omega$ on $\sP_{2,+}$,  define  $\widehat{\omega}$ as the weight on $\sP_{2,-}$ given by the density matrix 
\be\label{DefineHatOmega}
D_{\reallywidehat{\omega}}=|\fF(D_\omega^{1/2})|^2\;.
\ee
It follows that $\reallywidehat{\omega}(1)=\omega(1)$. If $\omega$ is a state, so is $\widehat\omega$.

\begin{theorem}[\bf REQ Uncertainty Principle]\label{thm:hbup1}
Let $\sP_\bullet$ be an irreducible subfactor planar algebra and $\varphi,\psi$ be faithful weights on $\sP_{2,\pm}$. 
Then for any state $\omega$ on $\sP_{2,+}$, 
\begin{equation*}\label{equ: relative entropyt UP}
\begin{aligned}
S(\omega\|\varphi)+S(\reallywidehat{\omega}\|\psi)\leq \log\|D_{\psi}^{-1}\|_\infty-\frac{1}{\delta^2}\tr_2(\log D_\varphi)-2\log\delta.
\end{aligned}
\end{equation*}
\end{theorem}

\begin{proof}
Note that  $\fF_{p,\varphi,\psi} (D_\omega^{1/2}  D_\varphi^{1/p-1/2})=\fF(D_\omega^{1/2}) D_\psi^{1/q-1/2}$. As $\norm{AB}_{p}=\norm{|A|\,B}_{p}$, using \eqref{DefineHatOmega}, we infer that $\fF_{p,\varphi,\psi} (D_\omega^{1/2}  D_\varphi^{1/p-1/2})$ and $D_{\reallywidehat{\omega}}^{1/2}D_\psi^{1/q-1/2}$ have the same $q$ norms.  
Define the function $f(p)$ as a picture, where $q=p/(p-1)$, 
\begin{equation*}
    f(p)=\norm{\raisebox{-.7cm}{
	\begin{tikzpicture}
	\path [fill=lightgray] (-.6,.5) rectangle (.7,2.1);
	\path [fill=white] (.2,2.1) rectangle (-.2,.5);
	\draw (.2,2.1)--(.2,.5);
	\draw (-.2,2.1)--(-.2,.5);
	\draw[fill=green] (-.25,.7) rectangle (.25,1.2);
	\draw [fill=yellow](-.15-.25,1.4) rectangle (.65-.25,1.9);
	\node at (0,1.65) {\tiny $D_\psi^{\frac{1}{q}-\frac{1}{2}}$};
	\node at (0,.95) {\tiny $D_{\reallywidehat{\omega}}^{\frac{1}{2}}$};
	\node at (-.35+.003,.95) {\tiny $\$$};
	\node at (-.5+.003,1.65) {\tiny $\$$};
	\end{tikzpicture}}
	\ }_q-
\left(
\left\Vert\,\raisebox{-.45cm}{\begin{tikzpicture}
\path [fill=lightgray] (0,-.25) rectangle (.1,.55);
\path (0,-.35) rectangle (-.1,.55);
\draw (0,-.25)--(0,.55);
\end{tikzpicture}}\,
\right\Vert^{-2/p}_{1}
  \left\Vert\,\raisebox{-.35cm}{
  \begin{tikzpicture}
  \path [fill=lightgray] (-.5,-.2) rectangle (.9,.7);
  \path [fill=white](.1,.7) rectangle (.6,-.2);
  \draw (.1,.7)--(.1,-.2);
  \draw (.6,.7)--(.6,-.2);
  \draw [fill=cyan] (-.2,0) rectangle (.8,.5);
  \node at (-.3+.003,.25) {\tiny $\$$};
  \node at (.35,.25) {\tiny$D_\psi^{\frac{1}{q}-\frac{1}{2}}$};
  \end{tikzpicture}}\,
\right\Vert_{\infty}\,\,
  \left\Vert\,\raisebox{-.35cm}{
  \begin{tikzpicture}
  \path [fill=lightgray] (-.3,-.2) rectangle (1,.7);
  \draw [fill=white](.1,.7)--(.1,0) arc[radius=.15, start angle=360, end angle=180]--(-.2,.7);
  \draw [fill=white](.6,-.2)--(.6,.5) arc[radius=.15, start angle=180, end angle=0]--(.9,-.2);
  \draw [fill=yellow] (-.05,0) rectangle (.75,.5);
  \node at (-.1+.003,.25) {\tiny $\$$};
  \node at (.35,.25) {\tiny$D_\varphi^{\frac{1}{2}-\frac{1}{p}}$};
  \end{tikzpicture}}\,
\right\Vert_{1}
\right)
	\norm{\raisebox{-.7cm}{
	\begin{tikzpicture}
	\path [fill=lightgray] (.2,2.1) rectangle (-.2,.5);
\draw (.2,2.1)--(.2,.5);
\draw (-.2,2.1)--(-.2,.5);
	\draw[fill=green] (-.25,.7) rectangle (.25,1.2);
	\draw [fill=cyan](-.15-.25,1.4) rectangle (.65-.25,1.9);
	\node at (0,1.65) {\tiny $D_\varphi^{\frac{1}{p}-\frac{1}{2}}$};
\node at (0,.95) {\tiny $D_\omega^{\frac{1}{2}}$};
\node at (-.35+.003,.95) {\tiny $\$$};
\node at (-.5+.003,1.65) {\tiny $\$$};
	\end{tikzpicture}}\ }_p
	\;.
\end{equation*}
The picture $f(p)$  is negative for $1\leq p\leq 2$, by Theorem~\ref{thm: relative HY}.  Also $f(2)=0$ by Plancherel's theorem, so the left derivative $f'_{-}(2)\geq0$. Then Theorem \ref{thm:hbup1} is a consequence of the expressions for the derivatives in the following lemma.
\end{proof}

\begin{lemma}\label{lem: diff}
For any weight $\omega$ on $\sP_{2,+}$, we have
\begin{equation*}
\begin{aligned}
\left. \frac{d}{dp}\norm{\raisebox{-.7cm}{
	\begin{tikzpicture}
	\path [fill=lightgray] (-.6,.5) rectangle (.7,2.1);
	\path [fill=white] (.2,2.1) rectangle (-.2,.5);
	\draw (.2,2.1)--(.2,.5);
	\draw (-.2,2.1)--(-.2,.5);
	\draw[fill=green] (-.25,.7) rectangle (.25,1.2);
	\draw [fill=yellow](-.15-.25,1.4) rectangle (.65-.25,1.9);
	\node at (0,1.65) {\tiny $D_\psi^{\frac{1}{q}-\frac{1}{2}}$};
	\node at (0,.95) {\tiny $D_{\reallywidehat{\omega}}^{\frac{1}{2}}$};
	\node at (-.35+.003,.95) {\tiny $\$$};
	\node at (-.5+.003,1.65) {\tiny $\$$};
	\end{tikzpicture}}
	\ }_q\right|_{p=2}=\raisebox{-1cm}{
	\begin{tikzpicture}
	\node at (0,.1)[right] {$-\displaystyle\frac{1}{4}\tr_2(D_{\reallywidehat{\omega}})^{1/2}\log\tr_2(D_{\reallywidehat{\omega}})$};
	\node at (0,-.8) [right]{$-\displaystyle\frac{1}{4\,\tr_2(D_{\reallywidehat{\omega}})^{1/2}}S({\reallywidehat{\omega}\|\psi})$};
	\end{tikzpicture}
}\;,
\end{aligned}
\end{equation*}
\begin{equation*}
\begin{aligned}
\left. \frac{d}{dp} \norm{\raisebox{-.7cm}{
	\begin{tikzpicture}
	\path [fill=lightgray] (.2,2.1) rectangle (-.2,.5);
\draw (.2,2.1)--(.2,.5);
\draw (-.2,2.1)--(-.2,.5);
	\draw[fill=green] (-.25,.7) rectangle (.25,1.2);
	\draw [fill=cyan](-.15-.25,1.4) rectangle (.65-.25,1.9);
	\node at (0,1.65) {\tiny $D_\varphi^{\frac{1}{p}-\frac{1}{2}}$};
\node at (0,.95) {\tiny $D_\omega^{\frac{1}{2}}$};
\node at (-.35+.003,.95) {\tiny $\$$};
\node at (-.5+.003,1.65) {\tiny $\$$};
	\end{tikzpicture}}\ }_p \right|_{p=2}&=\raisebox{-.95cm}{
\begin{tikzpicture}
\node at (0,.1)[right] {$-\displaystyle\frac{1}{4}\tr_2(D_\omega)^{1/2}\log\tr_2(D_\omega)$};
\node at (0,-.8) [right]{$+\displaystyle\frac{1}{4\,\tr_2(D_\omega)^{1/2}}S(\omega\|\varphi)$};
\end{tikzpicture}
}\;,
\end{aligned}
\end{equation*}
\begin{equation*}\label{equ: derivatve of constant}
\scalebox{.9}{$
\begin{aligned}
\left.\frac{d}{dp} 
\left(
\left\Vert\,\raisebox{-.45cm}{\begin{tikzpicture}
\path [fill=lightgray] (0,-.25) rectangle (.1,.55);
\path (0,-.35) rectangle (-.1,.55);
\draw (0,-.25)--(0,.55);
\end{tikzpicture}}\,
\right\Vert^{-2/p}_{1}
  \left\Vert\,\raisebox{-.35cm}{
  \begin{tikzpicture}
  \path [fill=lightgray] (-.5,-.2) rectangle (.9,.7);
  \path [fill=white](.1,.7) rectangle (.6,-.2);
  \draw (.1,.7)--(.1,-.2);
  \draw (.6,.7)--(.6,-.2);
  \draw [fill=cyan] (-.2,0) rectangle (.8,.5);
  \node at (-.3+.003,.25) {\tiny $\$$};
  \node at (.35,.25) {\tiny$D_\psi^{\frac{1}{q}-\frac{1}{2}}$};
  \end{tikzpicture}}\,
\right\Vert_{\infty}\,\,
  \left\Vert\,\raisebox{-.35cm}{
  \begin{tikzpicture}
  \path [fill=lightgray] (-.3,-.2) rectangle (1,.7);
  \draw [fill=white](.1,.7)--(.1,0) arc[radius=.15, start angle=360, end angle=180]--(-.2,.7);
  \draw [fill=white](.6,-.2)--(.6,.5) arc[radius=.15, start angle=180, end angle=0]--(.9,-.2);
  \draw [fill=yellow] (-.05,0) rectangle (.75,.5);
  \node at (-.1+.003,.25) {\tiny $\$$};
  \node at (.35,.25) {\tiny$D_\varphi^{\frac{1}{2}-\frac{1}{p}}$};
  \end{tikzpicture}}\,
\right\Vert_{1}
\right) \right|_{p=2}
=
\raisebox{-.9cm}{
\begin{tikzpicture}
\node (.25,.1) [right] {$\displaystyle\frac{1}{2}\log\delta-\displaystyle\frac{1}{4}\log\left\|D_{\psi}^{-1}\right\|_\infty$};
\node at (0,-.8) [right] {$+\displaystyle\frac{1}{4\,\delta^2}tr_2(\log D_\varphi).$};
\end{tikzpicture}
}\;.
\end{aligned}
$}
\end{equation*}
\end{lemma}

\section{QFA and Quantum Entanglement}
Here we convert our pictures to the style in~\cite{JL17,JaLiWo18}, which are not shaded. 
The Fourier transform of a multiple of the projection onto the zero-vector for the group $\mathbb{Z}_d$,  namely $d^{1/2}\ket{0}\bra{0}$,  is the identity,
\begin{equation}\label{FProj}
\scalebox{.8}{$
\mathfrak{F}\ 
\raisebox{-.5cm}{\color{blue}
\tikz{
\fqudit {0}{-3.5\nn}{1\mn}{.5\nn}{}
\fmeasure {0}{0}{1\mn}{.5\nn}{}
}}
=
\raisebox{-.5cm}{
\tikz{\color{blue}
\draw (0,0)--(0,1+1/6);
\draw (2/3,0)--(2/3,1+1/6);
}}
= d^{-1/2}\sum_{k\in \mathbb{Z}_d}
\raisebox{-.5cm}{
\tikz{\color{blue}
\fqudit {0}{-3.5\nn}{1\mn}{.5\nn}{\phantom{-}k}
\fmeasure {0}{0}{1\mn}{.5\nn}{-k}
}}\ 
$}
\;.
\end{equation}
One can identify a linear transformation $T$ on $\mathbb{C}^d$ as a vector  $\widehat{T}$ in $\mathbb{C}^d\otimes\mathbb{C}^d$, namely 
\[
\scalebox{.85}{$\raisebox{-.5cm}{
\tikz{\color{blue}
\fqudit {0}{-3.5\nn}{1\mn}{.5\nn}{\phantom{-}i}
\fmeasure {0}{0}{1\mn}{.5\nn}{-j}
}}\quad\longleftrightarrow\ 
\raisebox{-.5cm}{
\tikz{\color{blue}
\fqudit {2.5\mn}{1.5\nn}{1\mn}{.5\nn}{-j}
\fqudit {0}{0}{1\mn}{.5\nn}{\phantom{-}i}
\draw (-.5\nn,0)--(-.5\nn,.5) ;
\draw (-2.5\nn,0)--(-2.5\nn,.5) ;
}}$}
\;,
\ \text{ the picture for }
\ket{i}\bra{j}\longleftrightarrow\ket{-j,i}\;.
\]
Identifying \eqref{FProj} in this way gives the illustration of how  the Fourier transform  $\mathfrak{F}$ 
acts on product states. In particular, 
\[
\scalebox{.8}{$\mathfrak{F} \left(\right.\,  \raisebox{-.05cm}{$  
\tikz{\color{blue}
\fqudit{1\mn}{0}{.5\mn}{.5\nn}{}{}
\fqudit{2.5\mn}{0}{.5\mn}{.5\nn}{}{}
}$}\, \left.\right)
= d^{-1/2}\sum_{k\in \mathbb{Z}_d}
\raisebox{-.4cm}{
\tikz{\color{blue}
\fqudit {2.5\mn}{1.5\nn}{1\mn}{.5\nn}{-k}
\fqudit {0}{0}{1\mn}{.5\nn}{\phantom{-}k}
\draw (-.5\nn,0)--(-.5\nn,.5) ;
\draw (-2.5\nn,0)--(-2.5\nn,.5) ;
}}
=
\raisebox{-.35cm}
{\tikz{\color{blue}
\fdoublequdit{1\mn}{0}{.5\mn}{.5\nn}{}{}
}}
= d^{1/2} \ket{\text{Max}} 
$}\;.
\]
If $d=2$, then $k=-k\in\mathbb{Z}_d$, and $\ket{\text{Max}}$ is the usual Bell state.

A similar picture and vector $ \ket{\text{Max}}_{n}=\mathfrak{F}\ket{\vec 0}_{n}$ exists for  $n$ qudits  \cite{JaLiWo18}. 
This vector $ \ket{\text{Max}}_{n}$ generalizes the classical Bell state for $d=2$ for qubits to a maximally-entangled state for $n$ qudits of order $d$. 
Furthermore, one can apply $\mathfrak{F}$ to any product basis state; one thereby obtains a Max basis (of maximally entangled states). These are related to  a corresponding $n$-qudit GHZ basis~\cite{GHZ}. In \S4 of~\cite{JL18} one finds expressions for the various $\ket{\text{Max}}_{n}$ basis states, as well as their relation to, and expressions for, the $\ket{\text{GHZ}}_{n}$ basis states. Also see~\cite{LWJ17} for a Quon interpretation.

The quantum uncertainty principle QUP-1 in Theorem \ref{Thm:Fusion} for subfactors gives a lower bound for the entanglement entropy. The product ground state has minimal entanglement entropy. Each Max state has maximal entanglement entropy. In addition, we obtain a relative entropic uncertainty principle for quantum entanglement by applying Theorem~\ref{thm:hbup1}.

\section{Some Future Directions and Goals}\label{Sect:Questions}

We propose a few specific questions, but first cite some general directions that appear ripe for the development of QFA:

    $\bullet$ Establish additive combinatorics for quantum symmetries, such as for unitary modular tensor categories.
    
    $\bullet$ Establish a general theory for $\mathfrak{F}(\sP_{n,\pm})$, where  $\fF^{2n}=1$.

    $\bullet$  Understand Fourier analysis for infinite quantum symmetries within a  pictorial framework, such as surface algebras.
        
    $\bullet$ Seek further applications of QFA to quantum information. 

\phantomsection
\subsection{Questions on the Universal Inequality} 
There are three central problems for the classical Brascamp-Lieb inequality on $\bR$:

a. Can one find for which tuples of linear maps the best constant is finite?

b. Can the best constant be achieved? If so, it is proved in \cite{BCCT} that there exist Gaussian extremizers. 

c. Are all extremizers Gaussian?

Since all $\sP_{n,\pm}$, $n \in \mathbb{N}$, are finite dimensional, the best constant $C$ of the universal inequality is finite and the extremizer exists by the compactness.
We ask the following questions for the universal inequality:
\begin{question}
In which case, is the best constant achieved by tensor product of $n$-projections (the natural generalization of bi-projections)?
\end{question}

\begin{question}
If further, all the input belong to $\sP_{2,\pm}$, are the extremizers all bishifts of biprojections?
\end{question}

\begin{question}[\bf Finite abelian groups]
What are the best constants of the universal inequality on finite abelian groups?

\end{question}

\subsection{Questions on Subfactor Planar Algebras}
Suppose $\sP_{\bullet,\pm}$ is an irreducible subfactor planar algebra.
\begin{conjecture}
For any $\varepsilon>0$, there exists $\varepsilon'$ such that if $x\in \mathscr{P}_{2,\pm}$, 
$\|x-P\|_2<\varepsilon'$ and $\|\FS(x)-\lambda Q\|_2<\varepsilon'$, for some projections $P, Q$ and constant $\lambda$, then there is a biprojection $B$, such that $\|x-B\|<\varepsilon$.
\end{conjecture}

\begin{question}
Can one characterize the extremizers for the uncertainty principles on $n$-boxes, for $n\geq 3$?
\end{question}

\subsection{Block Renormalization Map and Quantum Central Limit Theorem}
The block map $B_\lambda$ is a composition of convolution and multiplication, 
\begin{align*}
B_\lambda\left(\hskip-.2cm
\scalebox{.8}{$\raisebox{-.3cm}{
\begin{tikzpicture}
\begin{scope}[shift={(-.5,0)}]
\path [fill=lightgray] (.05,-.2) rectangle (.4,.65);
\draw (.4,-.2)--(.4,.65);
\draw (.05,-.2)--(.05,.65);
	\draw [fill=white] (0,0) rectangle (.45,.45);
		\node at (.25,.25) {$x^{\phantom *}$};
		\node at (-.05,.25) {\tiny $\$$};
		\end{scope}
\end{tikzpicture}}$}\right)=\frac{\delta^2}{\norm{x}_2^2}\left(\frac{\lambda}{\norm{x}_1}
\scalebox{.8}{$
\raisebox{-.75cm}{
\begin{tikzpicture}
\path [fill=lightgray] (.05,-.25) rectangle (1.25,1.55);
\draw (.05,-.25)--(.05,1.55);
\draw (1.25,-.25)--(1.25,1.55);
\draw [fill=white] (.4,1.3) arc[radius=.15, start angle=180, end angle=90]--(.75,1.45) arc[radius=.15, start angle=90, end angle=0]--(.9,.85) arc[radius=.15, start angle=360, end angle=270]--(.55,.7) arc[radius=.15, start angle=270, end angle=180];
\begin{scope}[yscale=-1,shift={(0,-1.3)}]
\draw [fill=white] (.4,1.3) arc[radius=.15, start angle=180, end angle=90]--(.75,1.45) arc[radius=.15, start angle=90, end angle=0]--(.9,.85) arc[radius=.15, start angle=360, end angle=270]--(.55,.7) arc[radius=.15, start angle=270, end angle=180];
\end{scope}
\draw [fill=yellow] (0,.85) rectangle (.45,1.3);
\draw [fill=cyan] (.85,.85) rectangle (1.3,1.3);
\draw [fill=white] (.85,0) rectangle (1.3,.45);
\draw [fill=green] (0,0) rectangle (.45,.45);
\node at (.25,.25) {$x^*$};
\node at (1.1,.25) {$x^{\phantom *}$};
\node at (1.1,1.1) {$x^*$};
\node at (.25,1.1) {$x^{\phantom *}$};
\node at (.5,.25) {\tiny $\$$};
\node at (.8,.25) {\tiny $\$$};
\node at (.5,1.1) {\tiny $\$$};
\node at (.8,1.1) {\tiny $\$$};
\end{tikzpicture}
}
$}+\frac{(1-\lambda)}{\norm{x}_\infty}
\scalebox{.8}{$
\raisebox{-.75cm}{
	\begin{tikzpicture}
	\path [fill=lightgray] (.05,-.25) rectangle (1.25,1.55);
	\draw (.05,-.25)--(.05,1.55);
	\draw (1.25,-.25)--(1.25,1.55);
	\draw [fill=white] (.4,1.3) arc[radius=.15, start angle=180, end angle=90]--(.75,1.45) arc[radius=.15, start angle=90, end angle=0]--(.9,.85) arc[radius=.15, start angle=360, end angle=270]--(.55,.7) arc[radius=.15, start angle=270, end angle=180];
	\begin{scope}[yscale=-1,shift={(0,-1.3)}]
	\draw [fill=white] (.4,1.3) arc[radius=.15, start angle=180, end angle=90]--(.75,1.45) arc[radius=.15, start angle=90, end angle=0]--(.9,.85) arc[radius=.15, start angle=360, end angle=270]--(.55,.7) arc[radius=.15, start angle=270, end angle=180];
	\end{scope}
	\path [fill=white](.4,.85) rectangle (.9,.45);
	\draw (.4,.85) rectangle (.4,.45);
	\draw (.9,.85) rectangle (.9,.45);
	\draw [fill=green] (0,.85) rectangle (.45,1.3);
	\draw [fill=white] (.85,.85) rectangle (1.3,1.3);
	\draw [fill=cyan] (.85,0) rectangle (1.3,.45);
	\draw [fill=yellow] (0,0) rectangle (.45,.45);
	\node at (.25,1.1) {$x^*$};
	\node at (.25,.25) {$x^{\phantom *}$};
	\node at (1.1,.25) {$x^*$};
	\node at (1.1,1.1) {$x^{\phantom *}$};
	\node at (.5,.25) {\tiny $\$$};
	\node at (.8,.25) {\tiny $\$$};
	\node at (.5,1.1) {\tiny $\$$};
	\node at (.8,1.1) {\tiny $\$$};
	\end{tikzpicture}
}$}
\right).
\end{align*}
The limit points of the iteration of the block map are all biprojections for finite-index, irreducible subfactors \cite{JLWscm}. We regard this result as a quantum 2D central limit theorem.

\begin{conjecture}\label{Conj: 2D central limit}
For any $f\in L^{\infty}(\mathbb{R}^n) \cap L^{1}(\mathbb{R}^n) \cap L^{2}(\mathbb{R}^n)$,
$f$ converges either to $0$, or to a Gaussian function, under the action of the iteration of the block map $2^nB_{\lambda}$.
\end{conjecture}

\begin{conjecture}\label{Conj: entropy monoticity}
The Hirschman-Beckner entropy decreases under the action of the block map $2^nB_{\lambda}$, for any $f\in L^{\infty}(\mathbb{R}^n) \cap L^{1}(\mathbb{R}^n) \cap L^{2}(\mathbb{R}^n)$. The same question remains for finite cyclic groups.
\end{conjecture}

\section{Significance Statement}
Quantum methods give a new face to Fourier analysis. Pictorial intuition yields new insights, motivates new inequalities, and new uncertainty principles for different quantum symmetries. 

\section{Acknowledgement}
We are grateful for informative discussions with a number of colleagues, and especially thank~Dietmar~Bisch, Kaifeng~Bu, Pavel~Etingof, Xun~Gao, Leonard~Gross, J.~William~Helton, Vaughan~Jones, Christopher~King, Robert~Lin, Sebastian~Palcoux,  Terence~Tao, Zhenghan~Wang, Feng~Xu, and Quanhua~Xu. This research was supported in part by a grant TRT 0159 on mathematical picture language from the Templeton Religion Trust.  AJ was partially supported by ARO Grant W911NF1910302.  CJ was partially supported by NSFC Grant 11831006.   ZL was partially supported by Grant 100301004 from Tsinghua University.  JW was partially supported by NSFC Grant 11771413.


\begin{thebibliography}{100}
\bibitem{Tannaka}
Tannaka, T. (1939)
{\"{U}ber den Dualit\"{a}tssatz der nichtkommutativen topologischen Gruppen}.
\href{https://www.jstage.jst.go.jp/article/tmj1911/45/0/45_0_1/_article/-char/ja/}
{{\em Tohoku Mathematical Journal, First Series} 45:1--12.}

\bibitem{Krein}
Krein, M.~G. (1949) A principle of duality for a bicompact group and block
  algebra.
\href{ }{{\em Dokl. Akad. Nauk. SSSR} 69:725--728.}

\bibitem{Hirschman}
Hirschman, I. (1957) A note on entropy. 
\href{https://doi.org/10.2307/2372390}
{{\em Am. J.  Math.} 79:152--156.}


\bibitem{Everett}
Everett III, H. (1956) The Many-Worlds Interpretation of Quantum Mechanics: the theory of the universal wave function.\newline 
\href{http://inspirehep.net/record/1358321/files/dissertation.pdf}
{http://inspirehep.net/record/1358321/files/dissertation.pdf.}

\bibitem{Beckner}
Beckner, W. (1975) Inequalities in Fourier analysis.\newline 
\href{https://www.jstor.org/stable/1970980}
{{\em Ann.  Math.} 102(1):159--182.}

\bibitem{BiaMyc75}
Bia{\l}ynicki-Birula, I., Mycielski, J. (1975) Uncertainty relations for information entropy in wave mechanics. \href{https://link.springer.com/article/10.1007/BF01608825}{{\em Commun. Math. Phys.} 44(2):129--132.}

\bibitem{BrascampLieb}
Brascamp, H., Lieb, E. (1976) Best constants in Young's inequality, its converse, and its generalization to more than three functions. \href{https://www.sciencedirect.com/science/article/pii/0001870876901845}{{\em Adv. Math.} 20(2):151--173.}

\bibitem{Nelson1}
Nelson, E. (1966) A quartic interaction in two dimensions.
{ {\em Mathematical Theory of Elementary Particles}, pp. 69--73, Eds. 
	R.~Goodman and I.~Segal, MIT Press, Cambridge, MA.}

\bibitem{Glimm}
Glimm,~J. (1968) Boson Fields with nonlinear self interaction in two dimensions.\href{https://projecteuclid.org/download/pdf_1/euclid.cmp/1103840511}{{\em Commun. Math. Phys.}  8:12--25.}

\bibitem{Nelson2}
Nelson, E. (1973) The free Markoff field. 
\href{https://doi.org/10.1016/0022-1236(73)90025-6}
{{\em J. Funct. Anal.} 12(2):211--227.}


\bibitem{Federbush}
Federbush,~P.  (1969) Partially Alternate Derivation of a Result of Nelson. \href{https://deepblue.lib.umich.edu/bitstream/handle/2027.42/70132/JMAPAQ-10-1-50-1.pdf?sequence=2&isAllowed=y}{{\em J. Math. Phys.}, 10(1):50--52.}


\bibitem{Gross}
Gross,~L.  (1975) Hypercontractivity and logarithmic Sobolev inequalities for the Clifford-Dirichlet form. \href{https://projecteuclid.org/download/pdf_1/euclid.dmj/1077311187}{{\em Duke Math. J.} 43:383--396.}


\bibitem{Carlen-Lieb}
Carlen,~E. and Lieb,~E. (1992) Optimal Hypercontractivity for Fermi Fields and Related Non-Commutative Integration Inequalities. \href{https://link.springer.com/article/10.1007/BF02100048}{{\em Commun. Math. Phys.}, 155:27--46.}


\bibitem{Stein}
Stein, E. M., Murphy, T. S. (1993) Harmonic analysis: real-variable methods, orthogonality, and oscillatory integrals.
\newline
\href{https://www.jstor.org/stable/j.ctt1bpmb3s}
{{Princeton Mathematics Series, Volume 43. Princeton University Press.}}


\bibitem{Janson}
Janson, S. (1997) On Complex Hypercontractivity.\newline
\href{https://doi.org/10.1006/jfan.1997.3144}
{{\em J.  Funct.  Anal.} 151(1): 270--280.}

\bibitem{Xu}
Ricard,~\'E. and Xu,~Q.  (2016) A noncommutative martingale convexity inequality. \href{https://projecteuclid.org/download/pdfview_1/euclid.aop/1457960385}{{\em Annals of Probability}, 44(2):867--882.}


\bibitem{Junge}
Junge,~M., Palazuelos,~C., Parcet~J., Perrin~M.,  and Ricard~\'E. (2015) Hypercontractivity for free products. \href{https://smf.emath.fr/publications/hypercontractivite-pour-des-produits-libres}{{\em Ann. Sci. \'Ecole Norm. Sup}, 48:861–889.}

\bibitem{King}
Beigi,~S.  and King,~C. (2016) Hypercontractivity and the logarithmic Sobolev inequality for the completely bounded norm.
\href{https://doi.org/10.1063/1.4934729}{{\em J. Math. Phys.} 57:015206.}


\bibitem{JL18} 
Jaffe, A., Liu, Z. (2018) Mathematical Picture Language Program. \href{https://www.pnas.org/content/115/1/81}{{\em Proceedings of the National Academy of Sciences} 115(1):81--86.}

\bibitem{Liu16}
Liu, Z. (2016) Exchange relation planar algebras of small rank. \href{https://www.ams.org/tran/0000-000-00/S0002-9947-2016-06582-4/}{{\em Transactions of the American Mathematical Society}, 368(12):8303--8348.}

\bibitem{JLW16}
Jiang, C., Liu, Z., Wu, J. (2016) Noncommutative uncertainty principles. \href{https://doi.org/10.1016/j.jfa.2015.08.007}{{\em Journal of Functional Analysis}, 270:264--311.}

\bibitem{Bisch}
Bisch,~D. (1994) A note on intermediate subfactors.
\newline
\href{https://msp.org/pjm/1994/163-2/pjm-v163-n2-p01-p.pdf}
{{\em Pacific Journal of Mathematics}, 163:201--216.}


\bibitem{Lus87}
Lusztig,~G. (1987) Leading coefficients of character values of Hecke algebras. \href{http://www.ams.org.ezp-prod1.hul.harvard.edu/books/pspum/047.2/}{{\em Proc. Symp. Pure Math.}, 47(2):235--262.}

\bibitem{EGNO15}
Etingof, P., Gelaki, S., Nikshych, D., Ostrik, V. (2015) Tensor Categories. 
\href{https://bookstore.ams.org/surv-205/}
{AMS  Math. Surveys and Monographs,}Volume 205.

\bibitem{LPW19}
Liu,~Z., Palcoux,~S., and Wu,~J. (2019) Fusion Bialgebras and Fourier Analysis. \href{https://arXiv.org/abs/1910.12059}{arXiv:1910:12059.}

\bibitem{Jon83}
Jones,~V. (1983) Index for subfactors. \href{https://doi.org/10.1007/BF01389127}{{{\em Inventiones Mathematicae}, 72:1--25}.}

\bibitem{EvaKaw98}
Evans,~D., Kawahigashi,~Y. (1998) Quantum symmetries on operator algebras. \href{https://global.oup.com/academic/product/quantum-symmetries-on-operator-algebras-9780198511755}{Clarendon Press, Oxford, UK.}

\bibitem{Jon99}
Jones,~V. (1999) Planar algebras, {I}. \href{https://arxiv.org/abs/math/9909027}{arXiv:math/9909027v1.}


\bibitem{Atiyah88}
Atiyah,~M. (1988) Topological quantum field theory.
\newline 
\href{http://www.numdam.org/article/PMIHES_1988__68__175_0.pdf}
{{\em Publications Math{\'e}matiques de l'IH{\'E}S}. 68:175--186.}

\bibitem{JL17}
Jaffe, A. and Liu, Z. (2017) Planar para algebras and reflection positivity. \href{https://link.springer.com/article/10.1007/s00220-016-2779-4}{{\em Commun. Math. Phys.} 352(1):95--133.}

\bibitem{JL19}
Jaffe, A. and Liu Z. (2019) Reflection Positivity and Levin-Wen Models. \href{https://arxiv.org/abs/1901.10662}{arXiv:1901.10662, to appear in {\it Expositiones Mathematicae}.}


\bibitem{LW19}
Liu, Z. and Wu, J. (2019) Non-commutative R\'enyi Entropic Uncertainty Principles. \href{https://arxiv.org/abs/1904.04292}{arXiv:1904.04292.}


\bibitem{Muger03}
M{\"u}ger, M. (2003) From subfactors to categories and topology II: The quantum double of tensor categories and subfactors.\newline
\href{https://www.sciencedirect.com/science/article/pii/S0022404902002487}
{{\em Journal of Pure and Applied Algebra} 180(1-2): 159--219.}

\bibitem{Liu19}
Liu, Z. (2019) Quon language: surface algebras and Fourier duality. \href{https://link.springer.com/article/10.1007/s00220-019-03361-3}{{\em Commun. Math. Phys.}, 366(3):865--894.}


\bibitem{LWJ17}
Liu, Z., Wozniakowski, A., and Jaffe, A., (2017) Quon 3D language for quantum information. \href{https://www.pnas.org/content/pnas/114/10/2497.full.pdf}{{\em PNAS}, 114(10):2497--2502.}

\bibitem{KuVaes}
Kustermans, J., Vaes, S. (2000) Locally compact quantum groups. \href{https://doi.org/10.1016/S0012-9593(00)01055-7}{{\em Annales Scientifiques de l'Ecole Normale Superieure},  {33(6)}  837--934.}

\bibitem{Coo10}
Cooney, T. (2010) A Hausdorff-Young inequality for locally compact quantum groups. \href{https://www.worldscientific.com/doi/abs/10.1142/S0129167X10006677}{{\em International Journal of Mathematics}, 21(12):1619--1632.}

\bibitem{Cas13}
Caspers, M. (2013) The {$L^p$}-{F}ourier transform on locally compact quantum groups. \href{http://dx.doi.org/10.7900/jot.2010aug22.1949}{{\em Journal of Operator Theory}, 69:161--193.}


\bibitem{LW17}
Liu, Z., Wu, J. (2017) Uncertainty principles for Kac algebras. \href{http://dx.doi.org/10.1063/1.4983755}{{\em J. Math. Phys.}, 58:052102.}
\newblock

\bibitem{LWW}
Liu, Z., Wang, S., Wu, J.(2017) Young's inequality for locally compact quantum groups. \href{http://dx.doi.org/10.7900/jot.2016mar03.2104}{ {\em Journal of Operator Theory},  {77}:109--131.}

\bibitem{JLW18}
Jiang, C., Liu, Z., Wu, J. (2018) Uncertainty principles for locally compact quantum groups. \href{https://doi.org/10.1016/j.jfa.2017.09.010}{{{\em Journal of Functional Analysis}, {274}:2399--2445}.}

\bibitem{Araki76} Relative Entropy of States of von Neumann Algebras.
Araki, H. (1976) \href{https://doi.org/10.2977/prims/1195191148}{{\em Publ. RIMS, Kyoto Univ.} {11}:809--833.}

\bibitem{JaLiWo18}
Jaffe, A., Liu, Z., and Wozniakowski, A. (2018) Holographic software for quantum networks. \href{https://link.springer.com/article/10.1007/s11425-017-9207-3}{{\em Science China Mathematics} 61(4):593--626.}

\bibitem{GHZ}
Greenberger,~D.~M., Horne,~M.~A., and Zeilinger,~A. (1989) Going beyond Bell's theorem. {\em Bell's theorem, quantum theory, and conceptions of the universe}, Kafakos,~M., editor, Vol. 37 of  
\href{https://link.springer.com/chapter/10.1007/978-94-017-0849-4_10}
{\textit{Fundamental Theories of Physics}}, 
Springer, Heidelberg,

 
 \bibitem{BCCT}
Bennett,~J., Carbery,~A., Christ,~M., Tao,~T. (2008) The Brascamp-Lieb inequalities: finiteness, structure and extremals. 
\newline
\href{https://doi.org/10.1007/s00039-007-0619-6}
{{\em Geom. Funct. Anal.},  17:1343--1415.}


\bibitem{JLWscm}
Jiang,~C., Liu, Z., Wu, J. (2019) Block maps and Fourier analysis. \href{https://link.springer.com/article/10.1007/s11425-017-9263-7}{{\em Science China Mathematics}, 62(8):1585--1614.}

\end{thebibliography}
\end{document}